\documentclass[11pt,table]{article}

\usepackage{amsmath,amssymb,amsthm,amsfonts}
\usepackage{mathtools}
\usepackage[T1]{fontenc}
\usepackage[utf8]{inputenc}
\usepackage{caption}
\usepackage[english]{babel}
\usepackage[numbers,square]{natbib}
\usepackage{diagbox,slashbox}
\usepackage{array}
\usepackage{hhline}
\usepackage[dvipsnames,table]{xcolor}
\usepackage{tikz}
\usetikzlibrary{decorations.pathmorphing,arrows.meta,arrows,shapes.geometric,calc,patterns}
\usepackage[margin=.75in,bmargin=1.5in]{geometry}
\usepackage{soul}
\usepackage[colorlinks,urlcolor=blue,citecolor=,linkcolor=,hyperfootnotes=false]{hyperref}

\tikzset{
	meet/.style={regular polygon,regular polygon sides=3,fill,color=blue,scale=.5},
	joint/.style={regular polygon,regular polygon sides=3,shape border rotate=180,fill,color=red,scale=.5},
	both/.style={diamond,fill,color=purple,scale=.66},
	none/.style={circle,fill,scale=.5},
}

\newcommand{\PreserveBackslash}[1]{\let\temp=\\#1\let\\=\temp}
\newcolumntype{C}[1]{>{\PreserveBackslash\centering}p{#1}}
\newcolumntype{R}[1]{>{\PreserveBackslash\raggedleft}p{#1}}
\newcolumntype{L}[1]{>{\PreserveBackslash\raggedright}p{#1}}

\makeatletter
\def\@biblabel#1{[\textbf{#1}]}
\newcommand{\xMapsto}[2][]{\ext@arrow 0599{\Mapstofill@}{#1}{#2}}
\def\Mapstofill@{\arrowfill@{\Mapstochar\Relbar}\Relbar\Rightarrow}
\makeatother

\newlength\pad
\newcommand{\define}[1]{\emph{#1}}
\newcommand{\simp}[2]{#1:\,#2}
\newcommand{\K}{\mathbf{K}}
\newcommand{\M}{\mathbf{M}}

\newcommand{\C}[2]{\mathbf{Ch}(#1,#2)}
\newcommand{\CUP}[2]{\overset{#2}{\underset{\scriptstyle #1}{\bigcup}}}

\newcommand{\mth}[1]{\emph{#1}}
\newcommand{\nit}[1]{\emph{#1}}
\newcommand{\ident}{\mathsf{id}}
\newcommand{\eqy}{\!=\!}
\newtheorem{thm}{Theorem}
\newtheorem{lem}{Lemma}
\newtheorem{cor}{Corollary}
\newtheorem{prop}{Proposition}
\newtheorem{dfn}{Definition}
\newtheorem{qst}{Question}
\newtheorem*{clm*}{Claim}

\begin{document}
\bibliographystyle{plainnatbib}
\setlength{\bibsep}{6pt}

\renewcommand{\abstractname}{\vspace{-\baselineskip}}

\title{Kuratowski Monoids on Posets}
\author{Mark Bowron}

\maketitle\vspace{-14pt}

\begin{abstract}\noindent
Ciraulo \cite{2025_ciraulo} recently showed that Kuratowski's closure--complement problem for arbitrary powersets of
topological spaces extends constructively to the interior--pseudocomplement problem for arbitrary posets, using the
closure--interior problem for posets \cite{1993_garel_olivier} (CIP) as a natural starting point.
After a brief overview of CIP, we resolve two diagram-completeness problems left open by Ciraulo.
Finally, we study operator semigroups arising from a little-known 1941 theorem of Chittenden
\cite{1941_chittenden}, which generalizes CIP.

\medskip\noindent
\emph{Keywords}: closure, interior, locale, operator monoid, poset, pseudocomplement, semigroup, supplement.

\medskip\noindent
\emph{2020 Mathematics Subject Classification}: 06F05 (primary); 06A15, 06D22, 54A05 (secondary).

\end{abstract}

\section{Introduction}\label{sec:intro}

Kuratowski's closure--complement problem has numerous relatives in the literature.
Many, including the original formulations \cite{1955_kelley,1922_kuratowski},
concern functions on powersets of topological spaces \cite{2008_gardner_jackson,2010_sherman}.
Others extend the domain to Boolean algebras \cite{1974_gaida_eremenko}, lattices \cite{2017_wang_wang_tao_peng},
and locales \cite{2025_ciraulo,2000_he_zhang}.
Here we focus primarily on posets \cite{1941_chittenden,2025_ciraulo,1993_garel_olivier}.

\begin{dfn}\label{dfn:orderings}
Let \mth{$(L,\le)$} be a poset\nit.
Let \mth{$\ident$} denote the identity function and \mth{$c,i$} be general closure and interior operators
on \mth{$(L,\le)$,}
subject only to the usual inequalities \mth{$i\le\ident\le c$,} without assuming any duality
involving complementation\nit.
We call the monoid \mth{$\K$}\! in diagram \mth{$1$} of
Figure~\nit{\ref{fig:monoids}} \define{the Kuratowski lattice} and its elements
\define{Kuratowski operators.}
For any given \mth{$(L,\le,c,i)$} the quotient of \mth{$\K$} under equality is called its \define{Kuratowski monoid.}
We also use this term generically for quotients of other operator monoids besides \mth{$\K$.}

We call any set \mth{$C\subseteq\{\{o_1,o_2\}\mathop|o_1,o_2\in\K\}$} a \define{collapse} of
\mth{$\K$.} An element \mth{$x\in L$} \define{satisfies} \mth{$C$} \define{on} \mth{$\K$}
if and only if \mth{$o_1x=o_2x\iff\{o_1,o_2\}\in C$} for all \mth{$o_1,o_2\in\K$.}
The poset \mth{$L$} itself \define{satisfies} \mth{$C$} \define{on} \mth{$\K$} pointwise\nit.
Satisfaction of a partial order \mth{$P$} on \mth{$\K$} by elements and posets is defined
similarly\nit. A collapse \mth{$C$} of \mth{$\K$} is called \define{local (resp.\ global)} if it is satisfied on \mth{$\K$}
by some element of some poset \nit(\!\!\;resp\nit.\,by some poset\,\nit{).}
Local and global orderings \mth{$P$} of \mth{$\K$} are defined similarly\nit.
These notions apply to any set of poset operators\nit, not only to \mth{$\K$.}
For brevity\nit, the term \define{global collapse} can also refer to its associated
set of poset operators.\footnote{These terms were introduced by Bowron~\cite{2024_bowron}
to unify several previously used ones. For the topological version of $\K$,
\emph{properties}~\cite{2007_mccluskey_mcintyre_watson} and \emph{Kuratowski lattices}~\cite{2024_staiger_wagner}
referred to local orderings,
\emph{universals}~\cite{2007_mccluskey_mcintyre_watson} to global orderings, and
\emph{equivalence patterns}~\cite{2008_gardner_jackson} to local collapses.
For the poset version of~$\K$, the standard term \emph{congruences}~\cite{1993_garel_olivier} was used for global collapses.}
\end{dfn}

Let $X$ be a topological space.
When we drop the axiom $\overline\varnothing=\varnothing$, equivalently
$A^{\circ}=X\setminus\overline{X\setminus A}$ for $A=X$, two new Kuratowski monoids can occur
\cite[Figure~$2.2$]{2008_gardner_jackson}.
When we generalize $\mathcal{P}(X)$ to any poset, the dual relationship between closure and
interior completely disappears,
occasioning $12$ new monoids \cite[Part~II]{1993_garel_olivier}.

\begin{thm}\label{thm:monoids}
Exactly \mth{$18$} distinct Kuratowski monoids occur in posets \nit(Figure~\nit{\ref{fig:monoids}).}
\end{thm}

\begin{figure}
\renewcommand{\tabcolsep}{1pt}
\centering
\begin{tabular}{c}
\vrule width0pt height200pt\begin{tikzpicture}
[auto,
 block/.style={rectangle,draw=black,align=center,minimum width={40pt},
 minimum height={16pt},inner sep=3pt,scale=1},
 line/.style ={draw},
 scale=1]

	\node[block,anchor=south] (0) at (0,0) {\vphantom{$f$}$c\eqy i$};
	\node[block,anchor=south] (1) at (-4.5,1.5) {\vphantom{$f$}$c\eqy\ident$};
	\node[block,anchor=south] (2) at (4.5,1.5) {\vphantom{$f$}$\ident\eqy i$};
	\node[block,anchor=south] (3) at (-6,3) {\vphantom{$f$}$cic\eqy i$};
	\node[block,anchor=south] (4) at (6,3) {\vphantom{$f$}$c\eqy ici$};
	\node[block,anchor=south] (5) at (-6,4.75) {\vphantom{$f$}$ic\eqy i$};
	\node[block,anchor=south] (6) at (-4,4.75) {\vphantom{$f$}$ci\eqy i$};
	\node[block,anchor=south] (7) at (0,4.75) {\vphantom{$f$}$cic\eqy ici$};
	\node[block,anchor=south] (8) at (4,4.75) {\vphantom{$f$}$c\eqy ci$};
	\node[block,anchor=south] (9) at (6,4.75) {\vphantom{$f$}$c\eqy ic$};
	\node[block,anchor=south] (10) at (-5,6) {\vphantom{$f$}$ici\eqy i$};
	\node[block,anchor=south] (11) at (0,6) {\vphantom{$f$}$ic\eqy ici$\hspace{6pt}$ci\eqy ici$};
	\node[block,anchor=south] (13) at (5,6) {\vphantom{$f$}$c\eqy cic$};
	\node[block,anchor=south] (14) at (0,1.5) {\vphantom{$f$}$ic\eqy i,\,c\eqy ci$\hspace{6pt}$ci\eqy i,\,c\eqy ic$};
	\node[block,anchor=south] (16) at (0,3.5) {\vphantom{$f$}$ici\eqy i,\,c\eqy cic$};
	\node[block,anchor=south] (17) at (0,7.5) {$\K$};
	\node[anchor=south] (key) at (0,9.2) {\begin{minipage}[t]{284pt}\footnotesize 
In the upper diagram edges denote implication. The center column is self-dual with respect
to the partial order on $L$, so its nodes each occur on the powerset of some closure space; ED stands for
\textit{extremally disconnected} and OU
\textit{open unresolvable} \cite[Section~$2.1$]{2008_gardner_jackson}.\end{minipage}};

	\node[anchor=south,label={[label distance=-2pt]90:\scriptsize$1$}] at (17.90) {};
	\node[anchor=south,label={[label distance=-9.2pt]90:\tiny Kuratowski}] at (17.90) {};
	\node[anchor=south,label={[label distance=-9pt]90:\scriptsize$2$}] at (10.north) {};
	\node[anchor=south,label={[label distance=-9pt]90:\scriptsize$2$d}] at (13.north) {};
	\node[anchor=south,label={[label distance=-9.5pt]90:\scriptsize$3$}] at (11.130) {};
	\node[anchor=south,label={[label distance=-2pt]180:\tiny OU}] at (11.158) {};
	\node[anchor=south,label={[label distance=-9.5pt]90:\scriptsize$4$}] at (11.50) {};
	\node[anchor=south,label={[label distance=-2pt]0:\tiny ED}] at (11.22) {};
	\node[anchor=south,label={[label distance=-9pt]90:\scriptsize$5$}] at (5.140) {};
	\node[anchor=south,label={[label distance=-9pt]90:\scriptsize$5$d}] at (8.140) {};
	\node[anchor=south,label={[label distance=-9pt]90:\scriptsize$6$}] at (6.40) {};
	\node[anchor=south,label={[label distance=-9pt]90:\scriptsize$6$d}] at (9.40) {};
	\node[anchor=south,label={[label distance=-9pt]90:\raisebox{2pt}{\scriptsize$7$}}] at (7.north) {};
	\node[anchor=south,label={[label distance=-7pt]0:\tiny\raisebox{-18pt}{\begin{tabular}{l}ED and\\OU\end{tabular}}}] at (7.east) {};
	\node[anchor=south,label={[label distance=-9pt]90:\scriptsize$8$}] at (3.140) {};
	\node[anchor=south,label={[label distance=-9pt]90:\scriptsize$8$d}] at (4.40) {};
	\node[anchor=south,label={[label distance=-7pt]90:\scriptsize$9$}] at (16.north) {};
	\node[anchor=south,label={[label distance=-8pt]0:\tiny\begin{tabular}{l}non-top.\\w/duality\end{tabular}}] at (16.7) {};
	\node[anchor=south,label={[label distance=-9pt]90:\scriptsize$10$}] at (1.north) {};
	\node[anchor=south,label={[label distance=-9pt]90:\scriptsize$10$d}] at (2.north) {};
	\node[anchor=south,label={[label distance=-9pt]90:\scriptsize$11$}] at (14.130) {};
	\node[anchor=south,label={[label distance=-4pt]-90:\tiny\begin{tabular}{l}non-top.\\w/duality\end{tabular}}] at (14.191) {};
	\node[anchor=south,label={[label distance=-9pt]90:\scriptsize$12$}] at (14.50) {};
	\node[anchor=south,label={[label distance=-2pt]-90:\tiny partition}] at (14.350) {};
	\node[anchor=south,label={[label distance=-5pt]90:\scriptsize$13$}] at (0.north) {};
	\node[anchor=south,label={[label distance=-2pt]-90:\tiny discrete}] at (0.south) {};

	\draw[line] (0.north west) -- (1.south east);
	\draw[line] (0.north) -- (14.195);
	\draw[line] (0.north) -- (14.345);
	\draw[line] (0.north east) -- (2.south west);
	\draw[line] (1.north west) -- (3.south);
	\draw[line] (2.north east) -- (4.south);
	\draw[line] (3) -- (5);
	\draw[line] (3.north east) -- (6.south);
	\draw[line] (3.north east) -- (7.south);
	\draw[line] (4.north west) -- (7.south);
	\draw[line] (4.north west) -- (8.south);
	\draw[line] (4) -- (9);
	\draw[line] (5.north) -- (10.south);
	\draw[line] (5.north east) -- (11.west);
	\draw[line] (6.north) -- (10.south);
	\draw[line] (6.east) -- (11.337);
	\draw[line] (7.north) -- (11.203);
	\draw[line] (7.north) -- (11.337);
	\draw[line] (8.west) -- (11.203);
	\draw[line] (8.north) -- (13.south);
	\draw[line] (9.north west) -- (11.east);
	\draw[line] (9.north) -- (13.south);
	\draw[line] (10.east) -- (17.south west);
	\draw[line] (11.south) -- (11.north);
	\draw[line] (11.157) -- (17.south);
	\draw[line] (11.23) -- (17.south);
	\draw[line] (13.west) -- (17.south east);
	\draw[line] (14.15) -- (6.south);
	\draw[line] (14.15) -- (9.south);
	\draw[line] (14.15) -- (16.south);
	\draw[line] (14.south) -- (14.north);
	\draw[line] (14.165) -- (5.south);
	\draw[line] (14.165) -- (8.south);
	\draw[line] (14.165) -- (16.south);
	\draw[line] (16.north) -- (10.east);
	\draw[line] (16.north) -- (13.west);

\end{tikzpicture}\\
\end{tabular}

\bigskip\bigskip
\begin{tikzpicture}[scale=.4]
	\node (c) at (-1,3) [circle,fill,scale=.4,label={[label distance=-1pt]90:$c$}] {};
	\node (cic) at (-1,1) [circle,fill,scale=.4,label={[label distance=-4.5pt]135:$cic$}] {};
	\node (ic) at (-2,0) [circle,fill,scale=.4,label={[label distance=-1pt]180:$ic$}] {};
	\node (ci) at (0,0) [circle,fill,scale=.4,label={[label distance=-2pt]0:$ci$}] {};
	\node (ici) at (-1,-1) [circle,fill,scale=.4,label={[label distance=-4.5pt]-135:$ici$}] {};
	\node (i) at (-1,-3) [circle,fill,scale=.4,label={[label distance=-1pt]-90:$i$}] {};
	\node (id) at (2,0) [circle,fill,scale=.4,label={[label distance=-2pt]0:$\ident$}] {};
	\node (label) at (-1,-6) [label={[label distance=0pt]90:$1$}] {};
	\draw (i) -- (id) -- (c);
	\draw (i) -- (ici) -- (ci) -- (cic) -- (c);
	\draw (ici) -- (ic) -- (cic);
\end{tikzpicture}\hspace{\pad}
\begin{tikzpicture}[scale=.4]
	\node (c) at (-1,3) [circle,fill,scale=.4,label={[label distance=-1pt]90:$c$}] {};
	\node (cic) at (-1,1) [circle,fill,scale=.4,label={[label distance=-4.5pt]135:$cic$}] {};
	\node (ic) at (-2,0) [circle,fill,scale=.4,label={[label distance=-1pt]180:$ic$}] {};
	\node (ci) at (0,0) [circle,fill,scale=.4,label={[label distance=-2pt]0:$ci$}] {};
	\node (i) at (-1,-3) [circle,fill,scale=.4,label={[label distance=-1pt]-90:$ici=i$}] {};
	\node (id) at (2,0) [circle,fill,scale=.4,label={[label distance=-2pt]0:$\ident$}] {};
	\node (label) at (-1,-6) [label={[label distance=0pt]90:$2$}] {};
	\draw (i) -- (id) -- (c);
	\draw (i) -- (ci) -- (cic) -- (c);
	\draw (i) -- (ic) -- (cic);
\end{tikzpicture}\hspace{\pad}
\begin{tikzpicture}[scale=.4]
\node (c) at (-1,3) [circle,fill,scale=.4,label={[label distance=-1pt]90:$c$}] {};
\node (cic) at (-1,1) [circle,fill,scale=.4,label={[label distance=-1pt]180:$cic=ci$}] {};
\node (ci) at (-1,-1) [circle,fill,scale=.4,label={[label distance=-1pt]180:$ic=ici$}] {};
\node (i) at (-1,-3) [circle,fill,scale=.4,label={[label distance=-1pt]-90:$i$}] {};
\node (id) at (1,0) [circle,fill,scale=.4,label={[label distance=-2pt]0:$\ident$}] {};
\node (label) at (-1,-6) [label={[label distance=0pt]90:$3$}] {};
\draw (i) -- (id) -- (c);
\draw (i) -- (ci) -- (cic) -- (c);
\end{tikzpicture}\hspace{\pad}
\begin{tikzpicture}[scale=.4]
\node (c) at (-1,3) [circle,fill,scale=.4,label={[label distance=-1pt]90:$c$}] {};
\node (cic) at (-1,1) [circle,fill,scale=.4,label={[label distance=-1pt]180:$cic=ic$}] {};
\node (ci) at (-1,-1) [circle,fill,scale=.4,label={[label distance=-1pt]180:$ci=ici$}] {};
\node (i) at (-1,-3) [circle,fill,scale=.4,label={[label distance=-1pt]-90:$i$}] {};
\node (id) at (1,0) [circle,fill,scale=.4,label={[label distance=-2pt]0:$\ident$}] {};
\node (label) at (-1,-6) [label={[label distance=0pt]90:$4$}] {};
\draw (i) -- (id) -- (c);
\draw (i) -- (ci) -- (cic) -- (c);
\end{tikzpicture}\hspace{\pad}
\begin{tikzpicture}[scale=.4]
	\node (c) at (-1,2) [circle,fill,scale=.4,label={[label distance=-1pt]90:$c$}] {};
	\node (cic) at (-1,0) [circle,fill,scale=.4,label={[label distance=-1pt]180:$cic=ci$}] {};
	\node (i) at (-1,-2) [circle,fill,scale=.4,label={[label distance=-1pt]-90:$ic=ici=i$}] {};
	\node (id) at (1,0) [circle,fill,scale=.4,label={[label distance=-2pt]0:$\ident$}] {};
	\node (label) at (-1,-5) [label={[label distance=0pt]90:$5$}] {};
	\draw (i) -- (id) -- (c);
	\draw (i) -- (cic) -- (c);
\end{tikzpicture}

\vspace{-4pt}
\begin{tikzpicture}[scale=.4]
	\node (c) at (-1,2) [circle,fill,scale=.4,label={[label distance=-1pt]90:$c$}] {};
	\node (cic) at (-1,0) [circle,fill,scale=.4,label={[label distance=-1pt]180:$cic=ic$}] {};
	\node (i) at (-1,-2) [circle,fill,scale=.4,label={[label distance=-1pt]-90:$ci=ici=i$}] {};
	\node (id) at (1,0) [circle,fill,scale=.4,label={[label distance=-2pt]0:$\ident$}] {};
	\node (label) at (-1,-5) [label={[label distance=0pt]90:$6$}] {};
	\draw (i) -- (id) -- (c);
	\draw (i) -- (cic) -- (c);
\end{tikzpicture}\hspace{\pad}
\begin{tikzpicture}[scale=.4]
	\node (c) at (-1,2) [circle,fill,scale=.4,label={[label distance=-1pt]90:$c$}] {};
	\node (cic) at (-1,0) [circle,fill,scale=.4,label={[label distance=-1pt]180:$ic=ci$}] {};
	\node (i) at (-1,-2) [circle,fill,scale=.4,label={[label distance=-1pt]-90:$i$}] {};
	\node (id) at (1,0) [circle,fill,scale=.4,label={[label distance=-2pt]0:$\ident$}] {};
	\node (label) at (-1,-5) [label={[label distance=0pt]90:$7$}] {};
	\draw (i) -- (id) -- (c);
	\draw (i) -- (cic) -- (c);
\end{tikzpicture}\hspace{\pad}
\begin{tikzpicture}[scale=.4]
	\node (c) at (-1,1.5) [circle,fill,scale=.4,label={[label distance=-1pt]90:$c$}] {};
	\node (i) at (-1,-1.5) [circle,fill,scale=.4,label={[label distance=-1pt]-90:$cic=i$}] {};
	\node (id) at (-1,0) [circle,fill,scale=.4,label={[label distance=-2pt]180:$\ident$}] {};
	\node (label) at (-1,-5) [label={[label distance=0pt]90:$8$}] {};
	\draw (i) -- (id) -- (c);
\end{tikzpicture}\hspace{\pad}
\begin{tikzpicture}[scale=.4]
	\node (c) at (0,2) [circle,fill,scale=.4,label={[label distance=-1pt]90:$c=cic$}] {};
	\node (ic) at (-2,0) [circle,fill,scale=.4,label={[label distance=-1pt]180:$ic$}] {};
	\node (ci) at (2,0) [circle,fill,scale=.4,label={[label distance=-2pt]0:$ci$}] {};
	\node (i) at (0,-2) [circle,fill,scale=.4,label={[label distance=-1pt]-90:$ici=i$}] {};
	\node (id) at (0,0) [circle,fill,scale=.4,label={[label distance=-2pt]180:$\ident$}] {};
	\node (label) at (0,-5) [label={[label distance=0pt]90:$9$}] {};
	\draw (i) -- (id) -- (c);
	\draw (i) -- (ci) -- (c);
	\draw (i) -- (ic) -- (c);
\end{tikzpicture}\hspace{\pad}
\begin{tikzpicture}[scale=.4]
\node (c) at (-1,1.5) [circle,fill,scale=.4,label={[label distance=-1pt]90:$c=\ident$}] {};
\node (i) at (-1,-1.5) [circle,fill,scale=.4,label={[label distance=-1pt]-90:$cic=i$}] {};
\node (label) at (-1,-5) [label={[label distance=0pt]90:$10$}] {};
\draw (i)-- (c);
\end{tikzpicture}\hspace{\pad}
\begin{tikzpicture}[scale=.4]
	\node (c) at (-1,1.5) [circle,fill,scale=.4,label={[label distance=-1pt]90:$c=ci$}] {};
	\node (i) at (-1,-1.5) [circle,fill,scale=.4,label={[label distance=-1pt]-90:$ic=i$}] {};
	\node (id) at (-1,0) [circle,fill,scale=.4,label={[label distance=-2pt]180:$\ident$}] {};
	\node (label) at (-1,-5) [label={[label distance=0pt]90:$11$}] {};
	\draw (i) -- (id) -- (c);
\end{tikzpicture}\hspace{\pad}
\begin{tikzpicture}[scale=.4]
	\node (c) at (-1,1.5) [circle,fill,scale=.4,label={[label distance=-1pt]90:$c=ic$}] {};
	\node (i) at (-1,-1.5) [circle,fill,scale=.4,label={[label distance=-1pt]-90:$ci=i$}] {};
	\node (id) at (-1,0) [circle,fill,scale=.4,label={[label distance=-2pt]180:$\ident$}] {};
	\node (label) at (-1,-5) [label={[label distance=0pt]90:$12$}] {};
	\draw (i) -- (id) -- (c);
\end{tikzpicture}\hspace{\pad}
\begin{tikzpicture}[scale=.4]
	\node (id) at (0,0) [circle,fill,scale=.4,label={[label distance=-2pt]-90:$c=i$}] {};
	\node (label) at (0,-5) [label={[label distance=0pt]90:$13$}] {};
\end{tikzpicture}
\caption{All $18$ Kuratowski monoids \cite[Part~II]{1993_garel_olivier} on $(L,\le)$, up to order duality.}
\label{fig:monoids}
\end{figure}

Ciraulo~\cite[Lemma 2.2]{2025_ciraulo} proved that every inequality in~$\K$ is equivalent to some equation
in~$\K$.\footnote{Surprisingly, a direct proof of this elementary fact had never appeared in print---even for topological
spaces---until Ciraulo \cite{2025_ciraulo} established it for~$\K$ in 2025.
With some effort, it can be extracted from Table~$11$ in Bowron \cite{2024_bowron} for the topological case.}
It follows that every global ordering of~$\K$ is a global collapse,\footnote{The open problem of determining all \emph{local}
orderings and collapses of~$\K$ lies beyond the scope of this paper.} i.e.,
one of the $18$ Kuratowski monoids in Figure~\ref{fig:monoids}.

In 1941, Chittenden \cite{1941_chittenden} generalized the dual formulas %Kuratowski \cite{1922_kuratowski}
that underlie the closure-complement theorem, as follows.

\begin{thm}\label{thm:chittenden}
Let \mth{$s$} and \mth{$t$} be non-decreasing operators on a poset \mth{$(L,\le)$} such that
\mth{$s\le t$,} \mth{$s^3=s$,} and \mth{$t^3=t$.} Then\vspace{-4pt}
\[
(st)^2 = st\mbox{ and }(ts)^2 = ts.
\]
\end{thm}
\vspace{4pt}

\noindent
This raises the natural question below.
To answer it, in Section~\ref{sec:chittenden} we study the semigroups $s$ and $t$ generate
under the stated conditions. %For the cases $2\le m,n\le3$, we give their global collapses.
The paper concludes with a demonstration that Chittenden's case $(m=n=3)$
yields $46$ additional global collapses beyond the six in the topological closure-complement case.

\begin{qst}\label{qst:chittenden}
For \mth{$m,n\ge2$,} what is the least exponent \mth{$I(m,n)>1$} such that \mth{$(st)^{I(m,n)}=st$} for all
non-decreasing functions \mth{$s\le t$} such that \mth{$s^m=s$,} \mth{$t^n=t$} on an arbitrary poset \mth{$(L,\le)$?}
\end{qst}

Ciraulo \cite[\S\S$3$-$4$]{2025_ciraulo} left completeness of the interior--pseudocomplement and
localic closure--supplement diagrams open. In the next section we show that both of these Hasse diagrams are complete.

\section{The Interior--Pseudocomplement Problem}\label{sec:pseudo}

\subsection{The Partial Order on \texorpdfstring{$\M$}{M}}\label{subsec:hasse}

In a lattice $L$ with bottom element $\bot$, the element
$x^{*}:=\max\{y\in L\mathop|x\wedge y=\bot\}$, when it exists, is called the \define{pseudocomplement} of $x$.
When the \define{relative pseudocomplement} $x\rightarrow y:=\max\{z\in L\mathop|x\wedge z\le y\}$
exists for all $x,y\in L$, the lattice $L$ is bounded (i.e.,\ has both $\bot$ and $\top$),
and $x^{*}=x\rightarrow\bot$ holds for all $x\in L$.\footnote{Such an $L$ is called a \emph{Heyting algebra}.
In Section~\ref{subsec:localic} we focus on complete Heyting algebras---those with arbitrary meets and
joins---also known as locales.}
For fixed $a\in L$, the maps $f(x)=x\wedge a$ and $g(x)=a\rightarrow x$ form a Galois connection:
\[
f(x)\le y\iff x\le g(y).
\]
Setting $y=\bot$ gives $x\wedge a\le\bot\iff x\le a^{*}$. Since $a$ was arbitrary and $\wedge$ is commutative,
we have $a\le x^{*}\iff x\le a^{*}$ for all $a,x\in L$.  Because this property
involves only the order relation, it can be used to define a pseudocomplement operator ${-}:L\to L$ on an
arbitrary poset $(L,\le)$.

Ciraulo \cite{2025_ciraulo} proved that any such operator ${-}$ and any general interior $i$ on $(L,\le)$ together
generate at most $31$ distinct operators under composition.
Let $\M$ denote this monoid of (at most) $31$ operators.

\begin{prop}\label{prop:ip-monoid}
Figure~\nit{\ref{fig:ip-monoid}} gives the Hasse diagram of \mth{$\M$} where \mth{$b={-}i{-}$.}
\end{prop}

\begin{proof}
Since Ciraulo left it to the reader to verify the edges in Figure~\ref{fig:ip-monoid},
for completeness we do this now.

Edge $6$, $\ident\le{-}{-}$, follows directly from reflexivity of $\le$ and the definition of the pseudocomplement,
\[
x\le{-}({-}x)\ \Longleftrightarrow\ ({-}x)\le{-}x.
\]
Hence ${-}$ is antitone, since
$x\le y\ \Longrightarrow\ x\le y\le{-}({-}y)\ \Longleftrightarrow\ ({-}y)\le {-}x$.
Applying this to $\ident\le{-}{-}$ gives ${-}{-}{-}\le{-}$; right-multiplying $\ident\le{-}{-}$ by
${-}$ yields ${-}\le{-}{-}{-}$. Therefore ${-}{-}{-}={-}$, and
right-multiplying by ${-}$ shows that ${-}{-}$ is a closure operator on $L$.
By definition $b={-}i{-}$.
From $i\le\ident$ it follows that $i{-}\le{-}$, and hence ${-}{-}i{-}\le{-}$.
Left-multiplying by $i$ then applying antitonicity gives $bb=b$.
Since $\ident\le{-}{-}\le{-}i{-}=b$ and $b$ is monotone,
$b$ is a closure operator on $L$.

\begin{figure}
\renewcommand{\tabcolsep}{4pt}\hspace{60pt}
\begin{tikzpicture}[scale=2]
	\node[none] (i) at (0,0) [label={[label distance=-1pt]180:$i$}] {};
	\node[joint] (i--i) at (0,.5) [label={[label distance=-3pt]180:\raisebox{-16pt}{$i{-}{-}i$}}] {};
	\node[joint] (i--) at (-1,1) [label={[label distance=-1pt]180:$i{-}{-}$}] {};
	\node[joint] (ibi) at (0,1) [label={[label distance=-3pt]180:\raisebox{-16pt}{$ibi$}}] {};
	\node[joint] (--i) at (1,1) [label={[label distance=-3pt]0:\raisebox{14pt}{${-}{-}i$}}] {};
	\node[none] (ibi--) at (-1,1.5) [label={[label distance=-3pt]180:\raisebox{-16pt}{$ibi{-}{-}$}}] {};
	\node[none] (-b-) at (0,1.5) [label={[label distance=-4pt]180:\raisebox{8pt}{${-}b{-}$}}] {};
	\node[none] (--ibi) at (1,1.5) [label={[label distance=-3pt]0:\raisebox{-16pt}{${-}{-}ibi$}}] {};
	\node[both] (ib) at (-2,2) [label={[label distance=-1pt]180:$ib$}] {};
	\node[none] (--ibi--) at (0,2) [label={[label distance=-1pt]180:${-}{-}ibi{-}{-}$}] {};
	\node[both] (bi) at (2,2) [label={[label distance=-1pt]0:$bi$}] {};
	\node[meet] (--ib) at (-1,2.5) [label={[label distance=-3pt]180:\raisebox{8pt}{${-}{-}ib$}}] {};
	\node[meet] (bi--) at (1,2.5) [label={[label distance=-3pt]0:\raisebox{8pt}{$bi{-}{-}$}}] {};
	\node[meet] (bib) at (0,3) [label={[label distance=-3pt]180:\raisebox{8pt}{$bib$}}] {};
	\node[none] (b) at (0,4.5) [label={[label distance=-1pt]180:$b$}] {};
	\node[both] (id) at (3,1.5) [label={[label distance=-1pt]0:$\ident$}] {};
	\node[meet] (--) at (3,3) [label={[label distance=-1pt]0:${-}{-}$}] {};
	\node (1) at (0,4.1) [label={[label distance=-5pt]180:\scriptsize$1$}] {};
	\node (2) at (.65,4.1) [label={[label distance=-5pt]180:\scriptsize$2$}] {};
	\node (3) at (-.15,2.72) [label={[label distance=-5pt]180:\scriptsize$3$}] {};
	\node (4) at (.4,2.72) [label={[label distance=-5pt]180:\scriptsize$4$}] {};
	\node (5) at (2.25,2.72) [label={[label distance=-5pt]180:\scriptsize$5$}] {};
	\node (6) at (3,2.5) [label={[label distance=-5pt]180:\scriptsize$6$}] {};
	\node (7) at (-1.15,2.28) [label={[label distance=-5pt]180:\scriptsize$7$}] {};
	\node (8) at (-.7,2.29) [label={[label distance=-5pt]180:\scriptsize$8$}] {};
	\node (9) at (.9,2.31) [label={[label distance=-5pt]180:\scriptsize$9$}] {};
	\node (10) at (1.26,2.31) [label={[label distance=-5pt]180:\scriptsize$10$}] {};
	\node (11) at (-1.7,1.78) [label={[label distance=-5pt]180:\scriptsize$11$}] {};
	\node (12) at (-.55,1.78) [label={[label distance=-5pt]180:\scriptsize$12$}] {};
	\node (13) at (0.02,1.75) [label={[label distance=-5pt]180:\scriptsize$13$}] {};
	\node (14) at (0.45,1.95) [label={[label distance=-5pt]180:\scriptsize$14$}] {};
	\node (15) at (1.85,1.75) [label={[label distance=-5pt]180:\scriptsize$15$}] {};
	\node (16) at (-.97,1.26) [label={[label distance=-5pt]180:\scriptsize$16$}] {};
	\node (17) at (-.6,1.46) [label={[label distance=-5pt]180:\scriptsize$17$}] {};
	\node (18) at (-.68,1.2) [label={[label distance=-5pt]180:\scriptsize$18$}] {};
	\node (19) at (.92,1.2) [label={[label distance=-5pt]180:\scriptsize$19$}] {};
	\node (20) at (.48,1.08) [label={[label distance=-5pt]180:\scriptsize$20$}] {};
	\node (21) at (1.22,1.26) [label={[label distance=-5pt]180:\scriptsize$21$}] {};
	\node (22) at (2.74,1.2) [label={[label distance=-5pt]180:\scriptsize$22$}] {};
	\node (23) at (-.7,.8) [label={[label distance=-5pt]180:\scriptsize$23$}] {};
	\node (24) at (.22,.8) [label={[label distance=-5pt]180:\scriptsize$24$}] {};
	\node (25) at (.7,.9) [label={[label distance=-5pt]180:\scriptsize$25$}] {};
	\node (26) at (.22,.3) [label={[label distance=-5pt]180:\scriptsize$26$}] {};
	\node[meet] (keym) at (-3,4) [label={[label distance=1pt]0:\raisebox{2pt}{meet-irreducible}}] {};
	\node[both] (keyjm) at (-3,3.75) [label={[label distance=1pt]0:\raisebox{2pt}{join- and meet-irreducible}}] {};
	\node[joint] (keyj) at (-3,3.5) [label={[label distance=1pt]0:\raisebox{2pt}{join-irreducible}}] {};
	\draw (i) -- (i--i) -- (i--) -- (-b-);
	\draw (i--i) -- (ibi) -- (ibi--) -- (ib) -- (--ib) -- (bib) -- (b);
	\draw (i--i) -- (--i) -- (-b-) -- (--ibi--) -- (--ib);
	\draw (ibi) -- (--ibi) -- (--ibi--) -- (bi--) -- (bib);
	\draw (ibi--) -- (--ibi--);
	\draw (--ibi) -- (bi) -- (bi--);
	\draw (i--) -- (ibi--);
	\draw (--i) -- (--ibi);
	\draw (i) -- (id) -- (--) -- (b);
	\draw (-b-) -- (--);
	\draw[dotted] (ibi) -- (--);
	\draw[dashed,color=red] (i--i) -- (id);
	\draw[dotted] (ib) -- (bi--);
	\draw[dotted] (i--) -- (bi);
	\draw[dotted] (bi) -- (--ib);
	\draw[dotted] (--i) -- (ib);
	\path (id.north) edge[out=130,in=0,dotted] (bib.east);
\end{tikzpicture}\hspace{-78pt}
\raisebox{22pt}{\begin{tabular}{r|c|}\multicolumn{1}{c}{}&\multicolumn{1}{c}{\clap{$i{-}{-}i(1)>\ident(1)$}}\\\cline{2-2}
\raisebox{-4pt}{\begin{tikzpicture}[scale=.4]
\node (1) at (0,0) {\scriptsize$1$\!\!\!};
\node (2) at (0,2) {\scriptsize$2$\!\!\!};
\end{tikzpicture}}&\vrule width0pt height31pt
\begin{tikzpicture}
[myarrow/.style={-{Latex[sep=2pt]}},scale=.4]
\node (1) at (0,0) [circle,fill,scale=.4] {};
\node (01) at (2,0) [circle,fill,scale=.4] {};
\node (2) at (0,2) [circle,fill,scale=.4] {};
\node (02) at (2,2) [circle,fill,scale=.4] {};
\draw[myarrow] (1) -- (02) node[midway] {};
\end{tikzpicture}\\\hhline{~|-|}
\multicolumn{1}{c}{}
\end{tabular}}

\hspace{3.44pt}
\raisebox{200pt}{\begin{tabular}{r|c|}\multicolumn{1}{c}{}&\multicolumn{1}{c}{\clap{Fixpoint arrows are omitted.}}\\[6pt]
\multicolumn{1}{c}{}&\multicolumn{1}{c}{\clap{$\textcolor{cyan}{i{-}i(2)} > \textcolor{LimeGreen}{b{-}(2)}$}}\\\cline{2-2}
\raisebox{-4pt}{\begin{tikzpicture}[scale=.4]
\node (1) at (0,0) {\scriptsize$1$\!\!\!};
\node (2) at (0,2) {\scriptsize$2$\!\!\!};
\node (3) at (0,4) {\scriptsize$3$\!\!\!};
\end{tikzpicture}}&\vrule width0pt height54pt
\begin{tikzpicture}
[myarrow/.style={-{Latex[sep=2pt]}},scale=.4]
\node (1) at (0,0) [circle,fill,scale=.4] {};
\node (01) at (2,0) [circle,fill,scale=.4] {};
\node (001) at (4,0) [circle,fill,scale=.4] {};
\node (0001) at (6,0) [circle,fill,scale=.4] {};
\node (00001) at (8,0) [circle,fill,color=LimeGreen,scale=.4] {};
\node (000001) at (10,0) [circle,fill,scale=.4] {};
\node (2) at (0,2) [circle,fill,color=LimeGreen,scale=.4] {};
\node (02) at (2,2) [circle,fill,scale=.4] {};
\node (002) at (4,2) [circle,fill,color=cyan,scale=.4] {};
\node (0002) at (6,2) [circle,fill,scale=.4] {};
\node (00002) at (8,2) [circle,fill,scale=.4] {};
\node (000002) at (10,2) [circle,fill,scale=.4] {};
\node (3) at (0,4) [circle,fill,scale=.4] {};
\node (03) at (2,4) [circle,fill,scale=.4] {};
\node (003) at (4,4) [circle,fill,scale=.4] {};
\node (0003) at (6,4) [circle,fill,scale=.4] {};
\node (00003) at (8,4) [circle,fill,scale=.4] {};
\node (000003) at (10,4) [circle,fill,color=cyan,scale=.4] {};
\draw[myarrow] (1) -- (03) node[midway] {};
\draw[myarrow,color=LimeGreen] (2) -- (01) node[midway] {};
\draw[myarrow] (3) -- (01) node[midway] {};
\draw[myarrow,color=LimeGreen] (01) -- (003) node[midway] {};
\draw[myarrow] (02) -- (001) node[midway] {};
\draw[myarrow] (03) -- (001) node[midway] {};
\draw[myarrow,dashed,color=cyan] (002) -- (0001) node[midway] {};
\draw[myarrow,color=cyan] (0001) -- (00003) node[midway] {};
\draw[myarrow] (0002) -- (00001) node[midway] {};
\draw[myarrow,color=LimeGreen] (0003) -- (00001) node[midway] {};
\draw[myarrow,dashed] (00002) -- (000001) node[midway] {};
\end{tikzpicture}\\\hhline{~|-|}
\multicolumn{1}{c}{}\\
\end{tabular}}\hspace{-70pt}
\begin{tikzpicture}[scale=1.6]
\node[none] (i-) at (0,.75) [label={[label distance=-3pt]180:\raisebox{-16pt}{$i{-}$}}] {};
\node[joint] (-b) at (-1,1.5) [label={[label distance=-3pt]180:\raisebox{-16pt}{${-}b$}}] {};
\node[both] (-) at (-3,3) [label={[label distance=-1pt]180:${-}$}] {};
\node[none] (-bib) at (0,2.25) [label={[label distance=-3pt]0:\raisebox{4pt}{${-}bib$}}] {};
\node[joint] (ibi-) at (1,1.5) [label={[label distance=-3pt]0:\raisebox{-16pt}{$ibi{-}$}}] {};
\node[joint] (ib-) at (2,2.25) [label={[label distance=-3pt]0:\raisebox{-16pt}{$ib{-}$}}] {};
\node[both] (-ib) at (-1,3) [label={[label distance=-1pt]180:${-}ib$}] {};
\node[none] (-bi--) at (1,3) [label={[label distance=-1pt]180:${-}bi{-}{-}$}] {};
\node[both] (i-i) at (3,3) [label={[label distance=-1pt]0:$i{-}i$}] {};
\node[none] (-ibi--) at (0,3.75) [label={[label distance=-1pt]180:\raisebox{-11pt}{$bib{-}$}}] {};
\node[meet] (-bi) at (2,3.75) [label={[label distance=-3pt]0:\raisebox{8pt}{${-}bi$}}] {};
\node[meet] (b-) at (-1,4.5) [label={[label distance=-3pt]180:\raisebox{8pt}{$b{-}$}}] {};
\node[meet] (-ibi) at (1,4.5) [label={[label distance=-3pt]0:\raisebox{8pt}{${-}ibi$}}] {};
\node[none] (-i) at (0,5.25) [label={[label distance=-1pt]180:\raisebox{8pt}{$-i$}}] {};
\node (27) at (-.18,4.82) [label={[label distance=-5pt]180:\scriptsize$27$}] {};
\node (28) at (.5,4.82) [label={[label distance=-5pt]180:\scriptsize$28$}] {};
\node (29) at (-1.2,4.05) [label={[label distance=-5pt]180:\scriptsize$29$}] {};
\node (30) at (-.5,4.05) [label={[label distance=-5pt]180:\scriptsize$30$}] {};
\node (31) at (.8,4.05) [label={[label distance=-5pt]180:\scriptsize$31$}] {};
\node (32) at (1.5,4.05) [label={[label distance=-5pt]180:\scriptsize$32$}] {};
\node (33) at (-.5,3.43) [label={[label distance=-5pt]180:\scriptsize$33$}] {};
\node (34) at (.58,3.6) [label={[label distance=-5pt]180:\scriptsize$34$}] {};
\node (35) at (1.35,3.31) [label={[label distance=-5pt]180:\scriptsize$35$}] {};
\node (36) at (2.28,3.5) [label={[label distance=-5pt]180:\scriptsize$36$}] {};
\node (37) at (-1.2,1.94) [label={[label distance=-5pt]180:\scriptsize$37$}] {};
\node (38) at (-.28,2.4) [label={[label distance=-5pt]180:\scriptsize$38$}] {};
\node (39) at (.35,2.54) [label={[label distance=-5pt]180:\scriptsize$39$}] {};
\node (40) at (1.6,2.84) [label={[label distance=-5pt]180:\scriptsize$40$}] {};
\node (41) at (2.74,2.54) [label={[label distance=-5pt]180:\scriptsize$41$}] {};
\node (42) at (-.5,1.94) [label={[label distance=-5pt]180:\scriptsize$42$}] {};
\node (43) at (.8,1.94) [label={[label distance=-5pt]180:\scriptsize$43$}] {};
\node (44) at (1.6,2.02) [label={[label distance=-5pt]180:\scriptsize$44$}] {};
\node (45) at (-.18,1.18) [label={[label distance=-5pt]180:\scriptsize$45$}] {};
\node (46) at (.5,1.18) [label={[label distance=-5pt]180:\scriptsize$46$}] {};
\draw (i-) -- (-b) -- (-) -- (b-) -- (-i);
\draw (-b) -- (-bib) -- (-ib) -- (-ibi--) -- (b-);
\draw (i-) -- (ibi-) -- (-bib) -- (-bi--) -- (-ibi--) -- (-ibi) -- (-i);
\draw (ibi-) -- (ib-) -- (-bi--) -- (-bi) -- (-ibi);
\draw (ib-) -- (i-i) -- (-bi);
\draw[dotted] (-) -- (-ibi);
\draw[dashed,color=red] (-b) -- (i-i);
\draw[dashed,color=red] (-ib) -- (-bi);
\draw[dashed,color=red] (i-i) -- (b-);
\draw[dashed,color=red] (ibi-) -- (-);
\draw[dashed,color=red] (ib-) -- (-ib);
\end{tikzpicture}\hspace{-90pt}
\raisebox{190pt}{\begin{tabular}{r|c|}\multicolumn{1}{c}{}&\multicolumn{1}{c}{\clap{$\textcolor{cyan}{{-}ib(1)}>\textcolor{LimeGreen}{{-}bi(1)}$}}\\\cline{2-2}
\raisebox{-4pt}{\begin{tikzpicture}[scale=.4]
\node (1) at (0,0) {\scriptsize$1$\!\!\!};
\node (2) at (0,2) {\scriptsize$2$\!\!\!};
\node (3) at (0,4) {\scriptsize$3$\!\!\!};
\end{tikzpicture}}&\vrule width0pt height54pt
\begin{tikzpicture}
[myarrow/.style={-{Latex[sep=2pt]}},scale=.4]
\node (1) at (0,0) [circle,fill,color=cyan,scale=.4] {};
\node (01) at (2,0) [circle,fill,color=LimeGreen,scale=.4] {};
\node (001) at (4,0) [circle,fill,scale=.4] {};
\node (0001) at (6,0) [circle,fill,scale=.4] {};
\node (00001) at (8,0) [circle,fill,scale=.4] {};
\node (000001) at (10,0) [circle,fill,scale=.4] {};
\node (0000001) at (12,0) [circle,fill,scale=.4] {};
\node (2) at (0,2) [circle,fill,scale=.4] {};
\node (02) at (2,2) [circle,fill,scale=.4] {};
\node (002) at (4,2) [circle,fill,scale=.4] {};
\node (0002) at (6,2) [circle,fill,scale=.4] {};
\node (00002) at (8,2) [circle,fill,scale=.4] {};
\node (000002) at (10,2) [circle,fill,scale=.4] {};
\node (0000002) at (12,2) [circle,fill,color=LimeGreen,scale=.4] {};
\node (3) at (0,4) [circle,fill,scale=.4] {};
\node (03) at (2,4) [circle,fill,scale=.4] {};
\node (003) at (4,4) [circle,fill,scale=.4] {};
\node (0003) at (6,4) [circle,fill,scale=.4] {};
\node (00003) at (8,4) [circle,fill,scale=.4] {};
\node (000003) at (10,4) [circle,fill,color=cyan,scale=.4] {};
\node (0000003) at (12,4) [circle,fill,scale=.4] {};
\draw[myarrow,color=cyan] (1) -- (03) node[midway] {};
\draw[myarrow] (2) -- (03) node[midway] {};
\draw[myarrow] (3) -- (02) node[midway] {};
\draw[myarrow,dashed,color=cyan] (03) -- (002) node[midway] {};
\draw[myarrow,color=LimeGreen] (001) -- (0003) node[midway] {};
\draw[myarrow,color=cyan] (002) -- (0003) node[midway] {};
\draw[myarrow] (003) -- (0002) node[midway] {};
\draw[myarrow,dashed,color=teal] (0003) -- (00002) node[midway] {};
\draw[myarrow] (00001) -- (000003) node[midway] {};
\draw[myarrow,color=teal] (00002) -- (000003) node[midway] {};
\draw[myarrow] (00003) -- (000002) node[midway] {};
\draw[myarrow] (000001) -- (0000003) node[midway] {};
\draw[myarrow] (000002) -- (0000003) node[midway] {};
\draw[myarrow,color=LimeGreen] (000003) -- (0000002) node[midway] {};
\end{tikzpicture}\\\hhline{~|-|}
\multicolumn{1}{c}{}\\
\end{tabular}}\vspace{-146pt}

\hspace{20pt}
\raisebox{200pt}{\begin{tabular}{r|c|}\multicolumn{1}{c}{}&\multicolumn{1}{c}{\clap{$\textcolor{cyan}{ibi{-}(4)}>\textcolor{LimeGreen}{{-}(4)}$}}\\\cline{2-2}
\raisebox{-4pt}{\begin{tikzpicture}[scale=.4]
\node (1) at (0,0) {\scriptsize$1$\!\!\!};
\node (2) at (0,2) {\scriptsize$2$\!\!\!};
\node (3) at (0,4) {\scriptsize$3$\!\!\!};
\node (4) at (0,6) {\scriptsize$4$\!\!\!};
\end{tikzpicture}}&\vrule width0pt height76.5pt
\begin{tikzpicture}
[myarrow/.style={-{Latex[sep=2pt]}},scale=.4]
\node (1) at (0,0) [circle,fill,scale=.4] {};
\node (01) at (2,0) [circle,fill,scale=.4] {};
\node (001) at (4,0) [circle,fill,scale=.4] {};
\node (0001) at (6,0) [circle,fill,scale=.4] {};
\node (00001) at (8,0) [circle,fill,scale=.4] {};
\node (000001) at (10,0) [circle,fill,color=LimeGreen,scale=.4] {};
\node (0000001) at (12,0) [circle,fill,scale=.4] {};
\node (2) at (0,2) [circle,fill,scale=.4] {};
\node (02) at (2,2) [circle,fill,scale=.4] {};
\node (002) at (4,2) [circle,fill,scale=.4] {};
\node (0002) at (6,2) [circle,fill,scale=.4] {};
\node (00002) at (8,2) [circle,fill,scale=.4] {};
\node (000002) at (10,2) [circle,fill,scale=.4] {};
\node (0000002) at (12,2) [circle,fill,color=cyan,scale=.4] {};
\node (3) at (0,4) [circle,fill,scale=.4] {};
\node (03) at (2,4) [circle,fill,scale=.4] {};
\node (003) at (4,4) [circle,fill,scale=.4] {};
\node (0003) at (6,4) [circle,fill,scale=.4] {};
\node (00003) at (8,4) [circle,fill,scale=.4] {};
\node (000003) at (10,4) [circle,fill,scale=.4] {};
\node (0000003) at (12,4) [circle,fill,scale=.4] {};
\node (4) at (0,6) [circle,fill,color=cyan,scale=.4] {};
\node (04) at (2,6) [circle,fill,scale=.4] {};
\node (004) at (4,6) [circle,fill,scale=.4] {};
\node (0004) at (6,6) [circle,fill,scale=.4] {};
\node (00004) at (8,6) [circle,fill,color=LimeGreen,scale=.4] {};
\node (000004) at (10,6) [circle,fill,scale=.4] {};
\node (0000004) at (12,6) [circle,fill,scale=.4] {};
\draw[myarrow] (1) -- (04) node[midway] {};
\draw[myarrow] (2) -- (03) node[midway] {};
\draw[myarrow,color=cyan] (4) -- (01) node[midway] {};
\draw[myarrow,dashed] (03) -- (002) node[midway] {};
\draw[myarrow,dashed] (04) -- (002) node[midway] {};
\draw[myarrow,color=cyan] (001) -- (0004) node[midway] {};
\draw[myarrow] (002) -- (0003) node[midway] {};
\draw[myarrow] (004) -- (0001) node[midway] {};
\draw[myarrow,dashed] (0003) -- (00002) node[midway] {};
\draw[myarrow,dashed,color=cyan] (0004) -- (00002) node[midway] {};
\draw[myarrow] (00001) -- (000004) node[midway] {};
\draw[myarrow,color=cyan] (00002) -- (000003) node[midway] {};
\draw[myarrow,color=LimeGreen] (00004) -- (000001) node[midway] {};
\draw[myarrow,dashed,color=cyan] (000003) -- (0000002) node[midway] {};
\draw[myarrow,dashed] (000004) -- (0000002) node[midway] {};
\end{tikzpicture}\\\hhline{~|-|}
\multicolumn{1}{c}{}\\
\end{tabular}}\hspace{80pt}
\raisebox{200pt}{\begin{tabular}{r|c|}\multicolumn{1}{c}{}&\multicolumn{1}{c}{\clap{$\textcolor{cyan}{ib{-}(2)}>\textcolor{LimeGreen}{{-}ib(2)}$}}\\\cline{2-2}
\raisebox{-4pt}{\begin{tikzpicture}[scale=.4]
\node (1) at (0,0) {\scriptsize$1$\!\!\!};
\node (2) at (0,2) {\scriptsize$2$\!\!\!};
\node (3) at (0,4) {\scriptsize$3$\!\!\!};
\node (4) at (0,6) {\scriptsize$4$\!\!\!};
\end{tikzpicture}}&\vrule width0pt height76.5pt
\begin{tikzpicture}
[myarrow/.style={-{Latex[sep=2pt]}},scale=.4]
\node (1) at (0,0) [circle,fill,scale=.4] {};
\node (01) at (2,0) [circle,fill,scale=.4] {};
\node (001) at (4,0) [circle,fill,scale=.4] {};
\node (0001) at (6,0) [circle,fill,scale=.4] {};
\node (00001) at (8,0) [circle,fill,scale=.4] {};
\node (000001) at (10,0) [circle,fill,scale=.4] {};
\node (0000001) at (12,0) [circle,fill,color=LimeGreen,scale=.4] {};
\node (2) at (0,2) [circle,fill,color=cyan,scale=.4] {};
\node (02) at (2,2) [circle,fill,color=LimeGreen,scale=.4] {};
\node (002) at (4,2) [circle,fill,scale=.4] {};
\node (0002) at (6,2) [circle,fill,scale=.4] {};
\node (00002) at (8,2) [circle,fill,scale=.4] {};
\node (000002) at (10,2) [circle,fill,scale=.4] {};
\node (0000002) at (12,2) [circle,fill,scale=.4] {};
\node (3) at (0,4) [circle,fill,scale=.4] {};
\node (03) at (2,4) [circle,fill,scale=.4] {};
\node (003) at (4,4) [circle,fill,scale=.4] {};
\node (0003) at (6,4) [circle,fill,scale=.4] {};
\node (00003) at (8,4) [circle,fill,scale=.4] {};
\node (000003) at (10,4) [circle,fill,scale=.4] {};
\node (0000003) at (12,4) [circle,fill,scale=.4] {};
\node (4) at (0,6) [circle,fill,scale=.4] {};
\node (04) at (2,6) [circle,fill,scale=.4] {};
\node (004) at (4,6) [circle,fill,scale=.4] {};
\node (0004) at (6,6) [circle,fill,scale=.4] {};
\node (00004) at (8,6) [circle,fill,scale=.4] {};
\node (000004) at (10,6) [circle,fill,color=cyan,scale=.4] {};
\node (0000004) at (12,6) [circle,fill,scale=.4] {};
\draw[myarrow] (1) -- (04) node[midway] {};
\draw[myarrow] (3) -- (01) node[midway] {};
\draw[myarrow] (4) -- (01) node[midway] {};
\draw[myarrow] (01) -- (004) node[midway] {};
\draw[myarrow] (03) -- (001) node[midway] {};
\draw[myarrow] (04) -- (001) node[midway] {};
\draw[myarrow,dashed,color=teal] (002) -- (0001) node[midway] {};
\draw[myarrow,color=teal] (0001) -- (00004) node[midway] {};
\draw[myarrow] (0003) -- (00001) node[midway] {};
\draw[myarrow] (0004) -- (00001) node[midway] {};
\draw[myarrow,dashed] (00002) -- (000001) node[midway] {};
\draw[myarrow] (000001) -- (0000004) node[midway] {};
\draw[myarrow] (000003) -- (0000001) node[midway] {};
\draw[myarrow,color=LimeGreen] (000004) -- (0000001) node[midway] {};
\end{tikzpicture}\\\hhline{~|-|}
\multicolumn{1}{c}{}\\
\end{tabular}}\hspace{-34pt}
\raisebox{326pt}{\begin{tabular}{r|c|}\multicolumn{1}{c}{}&\multicolumn{1}{c}{\clap{${-}b>i{-}i$}}\\\cline{2-2}
\raisebox{-4pt}{\begin{tikzpicture}[scale=.4]
\node (1) at (0,0) {\scriptsize$1$\!\!\!};
\node (2) at (0,2) {\scriptsize$2$\!\!\!};
\end{tikzpicture}}&\vrule width0pt height31pt
\begin{tikzpicture}
[myarrow/.style={-{Latex[sep=2pt]}},scale=.4]
\node (1) at (0,0) [circle,fill,scale=.4] {};
\node (01) at (2,0) [circle,fill,scale=.4] {};
\node (2) at (0,2) [circle,fill,scale=.4] {};
\node (02) at (2,2) [circle,fill,scale=.4] {};
\draw[myarrow,dashed] (2) -- (01) node[midway] {};
\draw[myarrow] (1) -- (02) node[midway] {};
\end{tikzpicture}\\\hhline{~|-|}
\multicolumn{1}{c}{}
\end{tabular}}\vspace{-146pt}
\caption{The Hasse diagram of $\M$ \cite[\S$3$]{2025_ciraulo}. Dashed lines represent critical pairs.}
\label{fig:ip-monoid}
\end{figure}

Edges $1$, $22$, $28$, $32$, $37$, $38$, $42$ follow from \cite[Proposition~$2.1$]{2025_ciraulo} together with antitonicity.
Letting \simp{$n$}{$xRy$} mean edge $n$ is implied by $xRy$, we have:
\simp{$16$}{$(i\le ibi){-}{-}$}, \simp{$29$}{$(\ident\le b){-}$}, \simp{$30$}{${-}i[(\ident\le b)i{-}{-}]$},
\simp{$44$}{$(ibi\le ib){-}$}, \simp{$46$}{$(i\le ibi){-}$}.
Note that \simp{$34$}{$[{-}(22)]bi{-}{-}$} where numbers represent previously obtained edges.
The rest now follow in numerical order:
\simp{$2$}{${-}(37)$}, 
\simp{$3$}{$(2)ib$},
\simp{$4$}{$bi(2)$},
\simp{$5$}{$(37){-}$},
\simp{$7$}{$(6)ib$},
\simp{$8$}{${-}{-}i(4)$},
\simp{$9$}{${-}(34)$},
\simp{$10$}{$bi(6)$},
\simp{$11$}{$i(4)$},
\simp{$12$}{$(6)ibi{-}{-}$},
\simp{$13$}{${-}(30)$},
\simp{$14$}{${-}{-}ibi(6)$},
\simp{$15$}{${-}(32)$},
\simp{$17$}{$ibi(6)$},
\simp{$18$}{$(6)i{-}{-}$},
\simp{$19$}{${-}{-}i(6)$},
\simp{$20$}{$(6)ibi$},
\simp{$21$}{${-}(28)$},
\simp{$23$}{$i{-}{-}(22)$},
\simp{$24$}{$[i{-}(22)]{-}i$},
\simp{$25$}{$(22){-}{-}i$},
\simp{$26$}{$i[(6)i]$},
\simp{$27$}{${-}i(6)$},
\simp{$31$}{${-}ibi(6)$},
\simp{$33$}{${-}(8)$},
\simp{$35$}{${-}bi(6)$},
\simp{$36$}{$(6)i{-}i$},
\simp{$39$}{${-}(4)$},
\simp{$40$}{$(6)i{-}i{-}{-}$},
\simp{$41$}{$i{-}i(6)$},
\simp{$43$}{$(6)i{-}i{-}i{-}$},
\simp{$45$}{$(6)i{-}$}.

Suppose $i=\ident$. Since ${-}{-}{-}={-}$, this collapses
Figure~\ref{fig:ip-monoid} to just $3$ nodes and $1$ edge:

\vspace{-26pt}\hspace{400pt}
\begin{tikzpicture}[scale=.5]
	\node (--) at (0,2) [circle,fill,scale=.5,label={[label distance=-1pt]180:${-}{-}$}] {};
	\node (id) at (0,0) [circle,fill,scale=.5,label={[label distance=-1pt]180:$\ident$}] {};
	\node (-) at (2,1) [circle,fill,scale=.5,label={[label distance=-1pt]180:${-}$}] {};
	\draw (id) -- (--);
\end{tikzpicture}

\vspace{-4pt}\noindent
Since
${-}\le\ident\ \Longrightarrow\ {-}\le{-}{-}\ \Longleftrightarrow\ {-}{-}\le{-}\ \Longrightarrow\ \ident\le{-}$
and $\ident\not\le{-}$ holds for any pseudocomplement ${-}$, we conclude that for any $o_1,o_2$
from separate components in Figure~\ref{fig:ip-monoid}, $o_1\not\le o_2$ holds in general.

To rule out all further inequalities, we apply Bergman's \cite{1994_bergman} method to each component separately.

A node is \textit{join-irreducible} (resp.\ \textit{meet-irreducible}) if and only if it is not the join (resp.\ meet)
of a set of other nodes.
Bergman observed that nodes with exactly one descending edge are join-irreducible (the lower node of this edge prevents
the upper one from being a least upper bound)
and nodes with more than one descending edge---at least one of which is the only ascending
edge of its lower node---are not join-irreducible (the upper node is the join
of every pair consisting of the aforementioned lower node and one of the others).
The duals of these observations apply to meet-irreducibility.

Bergman's first step is to find all join- and meet-irreducible nodes.
We begin with the upper component.
The join $i$ of the empty set is not join-irreducible;
dually, the meet $b$ of the empty set is not meet-irreducible.
Nodes with exactly one descending edge are:
$ib$, $bi$, $\ident$, $i{-}{-}$, $ibi$, ${-}{-}i$, $i{-}{-}i$.
Nodes with multiple descending edges, at least one of which is the only ascending
edge of its lower node, are:
$b$, $bib$, ${-}{-}$, ${-}{-}ib$, $bi{-}{-}$.
Each of the remaining four nodes
${-}{-}ibi{-}{-}$, $ibi{-}{-}$, ${-}b{-}$, ${-}{-}ibi$
is the join of the set of lower nodes of its descending edges,
hence not join-irreducible.
Nodes with exactly one ascending edge are:
$bib$, ${-}{-}$, ${-}{-}ib$, $bi{-}{-}$, $ib$, $bi$, $\ident$.
Nodes with multiple ascending edges, at least one of which is the only descending
edge of its upper node, are:
$ibi{-}{-}$, ${-}{-}ibi$, $i{-}{-}i$, $i$.
Each of the remaining five nodes
${-}{-}ibi{-}{-}$, ${-}b{-}$, $i{-}{-}$, $ibi$, ${-}{-}i$
is the meet of the set of upper nodes of its ascending edges, hence not meet-irreducible.

The search for all join- and meet-irreducible nodes in the lower component is similar.

Let $J$ and $M$ be the sets of join- and meet-irreducible nodes, respectively, in Figure~\ref{fig:ip-monoid}.
By Lemma~$5$ in Bergman \cite{1994_bergman}, $(x,y)$
is a \textit{critical pair} if and only if $x\in\left(J\setminus{\downarrow}\{y\}\right)_{min}$ and
$y\in\left(M\setminus{\uparrow}\{x\}\right)_{max}$ where $S_{min}$ is the set
$\{x\in S\mathop|x<x'\,\Longrightarrow\,x'\not\in S\}$ of minimal nodes in $S$ and $S_{max}$ is defined similarly.

We observe in Figure~\ref{fig:ip-monoid} that for all $(x,y)\in J\times M$ such that $x\not\le y$,
minimality of $x$ or maximality of $y$ is violated if and only if a dotted or dashed line does not connect $x$ and $y$.
The dotted and dashed lines therefore represent all critical pairs.

Since every inequality $x\le y$ not implied by Figure~\ref{fig:ip-monoid} implies
$x'\le y'$ for some critical pair $(x',y')$, to prove completeness it suffices to
disprove every inequality represented by a critical pair.
The following implications eliminate the dotted lines, leaving only the inequalities represented by
the six dashed lines:

\bigskip\noindent
{
\renewcommand{\arraystretch}{1.2}
\renewcommand{\tabcolsep}{2pt}
\centering
\begin{tabular}{r@{\hspace{-.4pt}}clcr@{\hspace{-.4pt}}cl}
$(ibi{-}\le{-}bib\le{-})$&:&${-}(\ident\le bib)$&\kern20pt&$(i{-}i\le{-}bi\le b{-})$&:&${-}(i{-}{-}\le bi)$\\
$(ibi\le{-}{-}ibi\le{-}{-})$&:&${-}({-}\le{-}ibi)$&&$({-}ib\le{-}bi)$&:&${-}(bi\le{-}{-}ib)$\\
$(ibi{-}\le{-})$&:&$(ibi\le{-}{-}){-}$&&$({-}b\le ib{-}\le i{-}i)$&:&$({-}{-}i\le ib){-}$\\
$(ib{-}\le{-}bi{-}{-}\le{-}ib)$&:&${-}(ib\le bi{-}{-})$&&&&\\
\end{tabular}\vrule height0pt depth0pt width30pt 

}

\bigskip\noindent
The counterexamples in Figure~\ref{fig:ip-monoid} disprove all six inequalities that remain, completing the proof.
\end{proof}

We verified by computer that the $2\cdot\binom{6}{2}$ possible implications between the six remaining inequalities
are each disproved by some poset on $4$ points, thus none of our counterexamples are redundant.

\begin{figure}[!t]
\centering
\begin{tikzpicture}[scale=1]
\node[none] (i) at (3,-6) [label={[label distance=1pt]0:$i$}] {};
\node[joint] (i--) at (3,-5) [label={[label distance=-3pt]180:\raisebox{-15pt}{$i{-}{-}$}}] {};
\node[joint] (ibi) at (1.5,-5) [label={[label distance=-1pt]180:\raisebox{-15pt}{$ibi$}}] {};
\node[none] (ibi--) at (2.25,-4) [label={[label distance=-1pt]180:$ibi{-}{-}$}] {};
\node[both] (ib) at (2.25,-3) [label={[label distance=-3pt]0:\raisebox{-15pt}{$ib$}}] {};
\node[both] (bi) at (0,-4) [label={[label distance=-1pt]180:$bi$}] {};
\node[meet] (bi--) at (0,-2) [label={[label distance=-1pt]180:$bi{-}{-}$}] {};
\node[meet] (bib) at (0,-1) [label={[label distance=-1pt]180:$bib$}] {};
\node[none] (b) at (0,0) [label={[label distance=-1pt]180:$b$}] {};
\node[both] (id) at (4.5,-4) [label={[label distance=-1pt]0:$\ident$}] {};
\node[meet] (--) at (4.5,-3) [label={[label distance=-1pt]0:${-}{-}$}] {};
\draw (b) -- (bib) -- (bi--) -- (bi) -- (ibi) -- (i);
\draw (bib) -- (ib) -- (ibi--) -- (ibi);
\draw (bi--) -- (ibi--) -- (i--);
\draw (b) -- (--) -- (id) -- (i);
\draw (--) -- (i--) -- (i);
\draw[dotted] (--) -- (ibi);
\draw[dashed,color=red] (bi) -- (i--);
\draw[dashed,color=red] (ib) -- (bi);
\draw[dashed,color=red] (bib) -- (id);
\draw[dotted] (id) -- (i--);
\draw[dashed,color=red] (bi--) -- (ib);
\end{tikzpicture}\hspace{-50pt}
\begin{tikzpicture}[scale=1]
\node (label) at (-1.6,-1) [label={[label distance=2pt]180:$(\textsf{S}(L),\supseteq)$}] {};
\node[none] (-b) at (1.5,-6) [label={[label distance=2pt]180:${-}b$}] {};
\node[both] (-) at (0,-3) [label={[label distance=-1pt]180:${-}$}] {};
\node[joint] (-bib) at (3,-5) [label={[label distance=-1pt]0:${-}bib$}] {};
\node[joint] (ib-) at (3,-4) [label={[label distance=-1pt]0:$ib{-}$}] {};
\node[both] (-ib) at (1.5,-3) [label={[label distance=-1pt]180:${-}ib$}] {};
\node[both] (-bi) at (3,-2) [label={[label distance=-1pt]0:${-}bi$}] {};
\node[none] (bib-) at (1.5,-2) [label={[label distance=-1pt]0:$bib{-}$}] {};
\node[meet] (b-) at (0,-1) [label={[label distance=-1pt]180:$b{-}$}] {};
\node[meet] (-ibi) at (1.5,-1) [label={[label distance=-1pt]0:${-}ibi$}] {};
\node[none] (-i) at (0,0) [label={[label distance=-1pt]180:${-}i$}] {};
\draw (-i) -- (b-) -- (-) -- (-b);
\draw (-i) -- (-ibi) -- (bib-) -- (-ib) -- (-bib) -- (-b);
\draw (b-) -- (bib-) -- (ib-) -- (-bib);
\draw (-ibi) -- (-bi) -- (ib-);
\draw[dotted] (-) -- (-bib);
\draw[dotted] (-ibi) -- (-);
\draw[dotted] (b-) -- (-bi);
\draw[dotted] (-bi) -- (-ib);
\draw[dotted] (-ib) -- (ib-);
\end{tikzpicture}

\vspace{-10pt}\begin{tikzpicture}[scale=.98]
\node (c) at (3,6) [circle,fill,scale=.5,label={[label distance=1pt]0:$c$}] {};
\node (c--) at (3,5) [circle,fill,scale=.5,label={[label distance=-5pt]135:$c{-}{-}$}] {};
\node (cic) at (1.5,5) [circle,fill,scale=.5,label={[label distance=1pt]90:$cic$}] {};
\node (cic--) at (2.25,4) [circle,fill,scale=.5,label={[label distance=-1pt]180:$cic{-}{-}$}] {};
\node (ci) at (2.25,3) [circle,fill,scale=.5,label={[label distance=-1pt]0:\raisebox{10pt}{$ci$}}] {};
\node (ic) at (0,4) [circle,fill,scale=.5,label={[label distance=-1pt]180:$ic$}] {};
\node (ic--) at (0,2) [circle,fill,scale=.5,label={[label distance=-1pt]180:$ic{-}{-}$}] {};
\node (ici) at (0,1) [circle,fill,scale=.5,label={[label distance=-1pt]180:$ici$}] {};
\node (i) at (0,0) [circle,fill,scale=.5,label={[label distance=-1pt]180:$i$}] {};
\node (id) at (4.5,4) [circle,fill,scale=.5,label={[label distance=-1pt]0:$\ident$}] {};
\node (--) at (4.5,3) [circle,fill,scale=.5,label={[label distance=-1pt]0:${-}{-}$}] {};
\node[meet] (keym) at (-4,6.2) [label={[label distance=1pt]0:\raisebox{2pt}{meet-irreducible}}] {};
\node[both] (keyjm) at (-4,5.7) [label={[label distance=1pt]0:\raisebox{2pt}{join- and meet-irreducible}}] {};
\node[joint] (keyj) at (-4,5.2) [label={[label distance=1pt]0:\raisebox{2pt}{join-irreducible}}] {};
\draw (i) -- (ici) -- (ic--) -- (ic) -- (cic) -- (c);
\draw (ici) -- (ci) -- (cic--) -- (cic);
\draw (ic--) -- (cic--) -- (c--);
\draw (i) -- (--) -- (id) -- (c);
\draw (--) -- (c--) -- (c);
\end{tikzpicture}\hspace{30pt}
\begin{tikzpicture}[scale=.98]
\node (-i) at (1.5,6) [circle,fill,scale=.5,label={[label distance=2pt]180:${-}i$}] {};
\node (-) at (0,3) [circle,fill,scale=.5,label={[label distance=-1pt]180:${-}$}] {};
\node (-ici) at (3,5) [circle,fill,scale=.5,label={[label distance=-1pt]0:${-}ici$}] {};
\node (ci-) at (3,4) [circle,fill,scale=.5,label={[label distance=-1pt]0:$ci{-}$}] {};
\node (-ci) at (1.5,3) [circle,fill,scale=.5,label={[label distance=-1pt]180:${-}ci$}] {};
\node (-ic) at (3,2) [circle,fill,scale=.5,label={[label distance=-1pt]0:${-}ic$}] {};
\node (ici-) at (1.5,2) [circle,fill,scale=.5,label={[label distance=-1pt]0:$ici{-}$}] {};
\node (i-) at (0,1) [circle,fill,scale=.5,label={[label distance=-1pt]180:$i{-}$}] {};
\node (-cic) at (1.5,1) [circle,fill,scale=.5,label={[label distance=-5pt]-45:${-}cic$}] {};
\node (-c) at (0,0) [circle,fill,scale=.5,label={[label distance=-1pt]180:$-c$}] {};
\draw (-c) -- (i-) -- (-) -- (-i);
\draw (-c) -- (-cic) -- (ici-) -- (-ci) -- (-ici) -- (-i);
\draw (i-) -- (ici-) -- (ci-) -- (-ici);
\draw (-cic) -- (-ic) -- (ci-);
\end{tikzpicture}
\caption{The closure--supplement monoid \cite[\S$4$]{2025_ciraulo} on $(\textsf{S}(L),\subseteq)$.
Dashed lines represent critical pairs.}
\label{fig:localic}
\end{figure}

\subsection{The Pointfree Variant of Kuratowski's Closure--Complement Problem}\label{subsec:localic}

He and Zhang \cite{2000_he_zhang} showed that the usual closure~$c$, interior~$i$, and co-pseudocomplement~${-}$
(also called the \emph{supplement}) on the co-frame $(\textsf{S}(L),\subseteq)$ of sublocales
of a locale~$L$ generate at most $21$ distinct operators. Ciraulo \cite[\S$4$]{2025_ciraulo}
used Figure~\ref{fig:ip-monoid} to obtain the lower diagram in Figure~\ref{fig:localic} for this operator monoid.
Since the next proposition holds for all posets, the peculiarly localic content of Ciraulo's result
lies in Proposition~\ref{prop:locales}.

\begin{prop}\label{prop:localic}
The equation \mth{$i={-}{-}i$} collapses Figure~\nit{\ref{fig:ip-monoid}} to the upper diagram
in Figure~\nit{\ref{fig:localic}.}
\end{prop}

\noindent
\emph{Proof}.
This holds by left multiplication $(i=i{-}{-}i)$: $i(i={-}{-}i)$ and by the eight right multiplications:

\medskip\noindent
{
\renewcommand{\arraystretch}{1.2}
\renewcommand{\tabcolsep}{2pt}
\centering
\begin{tabular}{r@{\hspace{-.4pt}}clcr@{\hspace{-.4pt}}cl}
$(i{-}={-}b)$&:&$(i={-}{-}i){-}$&\kern20pt&$(ibi{-}{-}={-}{-}ibi{-}{-})$&:&$(i={-}{-}i)bi{-}{-}$\\
$(i{-}{-}={-}b{-})$&:&$(i={-}{-}i){-}{-}$&&$(i{-}i={-}bi)$&:&$(i{-}={-}b)i$\\
$(ib={-}{-}ib)$&:&$(i={-}{-}i)b$&&$(ibi{-}={-}bib)$&:&$(i{-}i={-}bi)b$\\
$(ibi={-}{-}ibi)$&:&$(i={-}{-}i)bi$&&$(ib{-}={-}bi{-}{-})$&:&$(i{-}i={-}bi){-}{-}$\\
\end{tabular}\vrule height0pt depth0pt width30pt

\vspace{-\baselineskip}\qed
}

\begin{prop}\label{prop:locales}
The closure operator \mth{$b:={-}i{-}$} on the frame \mth{$(\textsf{S}(L),\supseteq)$} sends each sublocale
\mth{$S\in\textsf{S}(L)$} to the largest open sublocale of \mth{$S$} in the co-frame \mth{$(\textsf{S}(L),\subseteq)$.}
\end{prop}

\begin{proof}
The co-frame $(\textsf{S}(L),\subseteq)$ satisfies $i={-}{-}i$ because open sublocales are complemented.
Ciraulo used this equation as follows.
Because the complement of a closed (resp.\ open) sublocale is open (resp.\ closed) we have
$c{-}i={-}i$ and $i{-}c={-}c$.
Hence $(c{-}\le{-}i)$: $c{-}(i\le\ident)$ and $({-}c\le i{-})$: $i{-}(\ident\le c)$.
Further, $({-}{-}i\le {-}c{-})$: ${-}(c{-}\le{-}i)$, $({-}c{-}\le i{-}{-})$: $({-}c\le i{-}){-}$,
and $(i{-}{-}\le i)$: $i({-}{-}\le\ident)$. Together these yield $i={-}{-}i\le {-}c{-}\le i{-}{-}\le i$.
Hence the operator $b$ coincides with the interior operator~$i$ on $(\textsf{S}(L),\subseteq)$.
\end{proof}

The next theorem settles a question left open by Ciraulo~\cite{2025_ciraulo}.

\begin{thm}\label{thm:localic}
The lower Hasse diagram in Figure~\nit{\ref{fig:localic}} is complete\nit.
\end{thm}

\begin{proof}
Ciraulo showed that the smallest dense sublocale $X_{\neg\neg}$ of $L:=(\omega+1)^{op}$ satisfies
${-}X_{\neg\neg}=L$. Since $X_{\neg\neg}$ is dense, it follows that
\[
ic{-}{-}(X_{\neg\neg})=ic({-}L)=ic(\{0\})=\{0\}\neq L=iL=ic(X_{\neg\neg}).\tag{1}
\]

A topological space $X$ satisfies the separation axiom $T_D$ if every $x\in X$ has an open neighborhood
$U$ such that $U\setminus\{x\}$ is open.
A well-known result in pointfree topology (e.g., \cite[Proposition~$2.3.1$]{2020_baboolal_picado_pillay_pultr})
states that distinct subspaces $Y,Z\subseteq X$ correspond to distinct sublocales $S_Y,S_Z\subseteq\Omega(X)$
if and only if $X$ is a $T_D$ space.
In this case the induced sublocales preserve inclusion:
\[
S_Y\subseteq S_Z\mbox{ if and only if }Y\subseteq Z.
\]
Consequently, for every subset $Y\subseteq X$ of a $T_D$ space,
\[
c(S_Y)=S_{\overline{Y}},\quad i(S_Y)=S_{Y^{\circ}},\quad\mbox{and}\quad{-}(S_Y)=S_{X\setminus Y}.
\]
Since $\mathbb{R}$ with the usual topology is $T_D$, these equations hold in $\Omega(\mathbb{R})$, giving
\[
ici(Y_{\mathbb{Q}})=Y_{\varnothing}\neq Y_{\mathbb{R}}=ic(Y_{\mathbb{Q}})=ic{-}{-}(Y_{\mathbb{Q}}),\tag{2}
\]

\vspace{-1.2\baselineskip}
\[
ici(Y_{(0,1)})=Y_{(0,1)}\neq Y_{[0,1]}=ci(Y_{(0,1)}),\tag{3}
\]
\[
i(Y_{\mathbb{R}\setminus\{0\}})=Y_{\mathbb{R}\setminus\{0\}}\neq Y_{\mathbb{R}}=
ici(Y_{\mathbb{R}\setminus\{0\}}).\tag{4}
\]

\medskip
It remains to show that $(1)$-$(4)$ disprove all critical inequalities in the lower diagram of Figure~\ref{fig:localic}.
To that end, we now apply Bergman's method to the upper diagram, beginning with the left component.

As before for~$\M$, the node $i$ is not join-irreducible and $b$ is not meet-irreducible.
Nodes with exactly one descending edge are $ib$, $bi$, $\ident$, $ibi$, $i{-}{-}$.
Nodes with multiple descending edges, at least one of which is the only ascending
edge of its lower node, are $b$, $bib$, $bi{-}{-}$, ${-}{-}$.
The remaining node $ibi{-}{-}$ is the join of the set of lower nodes of its descending edges,
and is therefore not join-irreducible.
Nodes with exactly one ascending edge are $bib$, $bi{-}{-}$, $ib$, ${-}{-}$, $bi$, $\ident$,
and those with multiple ascending edges---at least one of which is the only descending
edge of its upper node---are $ibi{-}{-}$, $ibi$, $i$.
The remaining node $i{-}{-}$ is the meet of the set of upper nodes of its ascending edges, and
hence not meet-irreducible.
The search for meet- and join-irreducible nodes in the right component is similar.
% Right component analysis (kept as a LaTeX comment for reference)
% For this component, ${-}b$ is not join-irreducible and ${-}i$ is not meet-irreducible.
% Nodes with exactly one descending edge are:
% ${-}bi$, ${-}$, ${-}ib$, $ib{-}$, ${-}bib$.
% Nodes with multiple descending edges, at least one of which is the only ascending
% edge of its lower node, are:
% ${-}i$, $b{-}$, ${-}ibi$, $bib{-}$.
% Nodes with exactly one ascending edge are:
% $b{-}$, ${-}ibi$, ${-}bi$, ${-}$, ${-}ib$.
% Nodes with multiple ascending edges, at least one of which is the only descending
% edge of its upper node, are:
% $ib{-}$, ${-}bib$, ${-}b$.
% The remaining node $(bib{-})$ is the meet of the set of upper nodes of its ascending edges,
% hence not meet-irreducible.
The first seven implications below reduce the eleven critical inequalities
in Figure~\ref{fig:localic} to the four dashed lines.
The subsequent implications show that strictness of four edges implies completeness of the upper diagram;
by order duality, if the duals of these four edges are strict in $(\textsf{S}(L),\subseteq)$,
then the lower diagram is complete. Equivalently, relations $(1)$-$(4)$ suffice.

\medskip\noindent
{
\renewcommand{\arraystretch}{1.2}
\renewcommand{\tabcolsep}{2pt}
\centering
\begin{tabular}{r@{\hspace{-.4pt}}clcr@{\hspace{-.4pt}}cl}
$(bi\le{-}{-}ib=ib)$&:&${-}({-}ib\le{-}bi)$,&\kern20pt&$(ibi={-}{-}ibi\le{-}{-})$&:&${-}({-}\le{-}ibi)$,\\
$(ib={-}{-}ib\le{-}ib{-}=bi{-}{-})$&:&${-}(ib{-}\le{-}ib)$,&&$({-}bib=ibi{-}\le{-})$&:&$(ibi\le{-}{-}){-}$,\\
$(i{-}{-}\le i\le bi)$&:&$i(i{-}{-}\le\ident)$,&&$(\ident\le{-}{-}\le bib)$&:&${-}({-}bib\le{-})$,\\
$(i{-}{-}={-}b{-}\le bi)$&:&${-}({-}bi\le b{-})$,&&&&\\
\end{tabular}\hspace{30pt}

}

\medskip
$(bi=bi{-}{-})$: $b(i{-}{-}\le bi)$, $(bib=bi{-}{-})$: $b(ib\le bi{-}{-})$,
$(ib=bib)$: $(bi\le ib)b$, $(b=bib)$: $b(\ident\le bib)$.
\end{proof}

\section{Chittenden Semigroups}\label{sec:chittenden}

In this section we study the operator semigroup generated by the maps $s$ and $t$ in Question~\ref{qst:chittenden}.

\begin{dfn}\label{dfn:chittenden}
Let \mth{$\{s,t\}^+$} denote the set of all nonempty finite words over the alphabet \mth{$\{s,t\}$.}
For integers \mth{$2\le m\le n$,} define \mth{$\C mn$} to be the
quotient \mth{$\{s,t\}^+/\equiv_{m,n}$} where \mth{$w_1\equiv_{m,n}w_2$} if and only if \mth{$w_1$} and
\mth{$w_2$} act identically \nit(\!\!\;as operator compositions\nit)\! on every poset \mth{$(L,\le)$} for all
order-preserving endofunctions \mth{$s,t:L\to L$} satisfying \mth{$s\le t$,} \mth{$s^m=s$,} and \mth{$t^n=t$.}
The induced order on \mth{$\C mn$} is defined by
\[
[w_1]\le [w_2]\iff\big(w_1(x)\le w_2(x)\mbox{ for all such }(L,\le,s,t)\mbox{ and all }x\in L\big)
\]
where \mth{$[w]$} denotes the equivalence class of \mth{$w$;}
for brevity we omit the square brackets in the sequel\!\;\nit.
\end{dfn}

The following lemma and corollary are immediate.

\begin{lem}\label{lem:symmetry}
Let \mth{$2\le m\le n$.}
Because order-preserving endofunctions \mth{$s\le t$} on a poset \mth{$(L,\le)$}\!\!\; correspond to
order-preserving endofunctions \mth{$t\le s$} on the dual poset \mth{$(L,\ge)$,} we have
\[
w_1\equiv_{m,n}w_2\ \Longleftrightarrow\ ({\sim}w_1)\equiv_{n,m}({\sim}w_2)
\]
for all \mth{$w_1,w_2\in\{s,t\}^+$} where \mth{${\sim}:\{s,t\}^+\to\{s,t\}^+$} denotes bitwise
negation \nit(interchanging \mth{$s$} and \mth{$t$);} for example\nit,
\mth{${\sim}(sts)=tst$.}
\end{lem}

\begin{cor}\label{cor:symmetry}
\mth{$\C mn\cong\C nm$} for all \mth{$2\le m\le n$.}
\end{cor}

\begin{lem}\label{lem:key}
Let \mth{$2\le m\le n$,} \mth{$d=\gcd(m-1,n-1)$,} and \mth{$\ell=\operatorname{lcm}(m-1,n-1)$.}
In~\mth{$\C mn$} we have\nit:

\medskip\noindent
\nit{(i)}~\mth{$s^{k\ell+1}=s$} and \mth{$t^{k\ell+1}=t$} for all \mth{$k\ge1$,}

\medskip\noindent
\nit{(ii)}~\mth{$swt=stt$} and \mth{$tws=tts$} for all \mth{$w\in\{s,t\}^{\ell+1}$,}

\medskip\noindent
\nit{(iii)}~\mth{$s^kt=s^{k-1}t^2=\cdots=st^k$} and \mth{$t^ks=t^{k-1}s^2=\cdots=ts^k$} for all \mth{$k\ge1$,}

\medskip\noindent
\nit{(iv)}~\mth{$s^{kd+1}t=st$} and \mth{$t^{kd+1}s=ts$} for all \mth{$k\ge1$.}
\end{lem}

\begin{proof}
Let $w\in\{s,t\}^{\ell+1}$. The assumptions of Definition~\ref{dfn:chittenden} imply (i) and (ii) as follows:

\medskip\noindent
{
\centering
\begin{tabular}{l}
$s^{k\ell+1}=s^{\frac{k(n-1)}{d}(m-1)+1}=s\mbox{\ \ and\ \ }t^{k\ell+1}=t^{\frac{k(m-1)}{d}(n-1)+1}=t$,\\[6pt]
$swt\le st^{\ell+1}t=stt=s^{2\ell+1}tt\le swt^{\ell+1}=swt$,\\[6pt]
$tws\le tt^{\ell+1}s=tts=tts^{3\ell+1}\le t^{\ell+1}ws^{\ell+1}=tws$.\\
\end{tabular}

}

\medskip\noindent
From these relations we have $sst=ss^{\ell+1}t=stt$ and $tss=ts^{\ell+1}s=tts$. To obtain~(iii),
repeatedly substitute $stt$ for $sst$ in $s^kt$ and $tts$ for $tss$ in $ts^k$,
performing the substitution $k-1$ times in each case.

Finally, we prove~(iv) by induction on $k$.
For $s^{d+1}t=st$ to hold, it suffices to find integers
$j,k\equiv1\pmod{m-1}$ and $r,s\equiv1\pmod{n-1}$ such that
\[
st=s^jt\le st^dt^r=st^{d+1}=s^{d+1}t=s^ks^dt\le st^s=st,
\]

\begin{figure}
\centering
\begin{tikzpicture}[scale=.5]
	\node (t) at (-1,3) [circle,fill,scale=.5,label={[label distance=-1pt]180:$t$}] {};
	\node (s) at (-1,-3) [circle,fill,scale=.5,label={[label distance=-1pt]180:$s$}] {};
	\draw (s) -- (t);
\end{tikzpicture}\hspace{6pt}
\begin{tikzpicture}[scale=.5]
	\node (t) at (-1,3) [circle,fill,scale=.5,label={[label distance=-1pt]180:$tt$}] {};
	\node (st) at (-2,0) [circle,fill,scale=.5,label={[label distance=-1pt]180:$st$}] {};
	\node (ts) at (0,0) [circle,fill,scale=.5,label={[label distance=-1pt]0:$ts$}] {};
	\node (s) at (-1,-3) [circle,fill,scale=.5,label={[label distance=-1pt]180:$ss$}] {};
	\draw (s) -- (st) -- (t);
	\draw (s) -- (ts) -- (t);
\end{tikzpicture}\hspace{6pt}
\begin{tikzpicture}[scale=.5]
	\node (t) at (-1,3) [circle,fill,scale=.5,label={[label distance=-1pt]180:$ttt$}] {};
	\node (tst) at (-1,1) [circle,fill,scale=.5,label={[label distance=-4pt]135:$tst$}] {};
	\node (st) at (-2,0) [circle,fill,scale=.5,label={[label distance=-1pt]180:$sst$}] {};
	\node (ts) at (0,0) [circle,fill,scale=.5,label={[label distance=-1pt]0:$tts$}] {};
	\node (sts) at (-1,-1) [circle,fill,scale=.5,label={[label distance=-3pt]-135:$sts$}] {};
	\node (s) at (-1,-3) [circle,fill,scale=.5,label={[label distance=-1pt]-180:$sss$}] {};
	\draw (sts) -- (ts) -- (tst) -- (t);
	\draw (s) -- (sts) -- (st) -- (tst);
\end{tikzpicture}\hspace{6pt}
\begin{tikzpicture}[scale=.5]
	\node (t) at (-1,3) [circle,fill,scale=.5,label={[label distance=-1pt]180:$tttt$}] {};
	\node (tst) at (-1,1) [circle,fill,scale=.5,label={[label distance=-4pt]135:$ttst$}] {};
	\node (st) at (-2,0) [circle,fill,scale=.5,label={[label distance=-1pt]180:$ssst$}] {};
	\node (ts) at (0,0) [circle,fill,scale=.5,label={[label distance=-1pt]0:$ttts$}] {};
	\node (sts) at (-1,-1) [circle,fill,scale=.5,label={[label distance=-3pt]-135:$ssts$}] {};
	\node (s) at (-1,-3) [circle,fill,scale=.5,label={[label distance=-1pt]180:$ssss$}] {};
	\draw (sts) -- (ts) -- (tst) -- (t);
	\draw (s) -- (sts) -- (st) -- (tst);
\end{tikzpicture}
\hspace{1pt}\raisebox{46pt}{$\cdots$}
\raisebox{-4pt}{\begin{tikzpicture}[scale=.5]
	\node (t) at (-1,3) [circle,fill,scale=.5,label={[label distance=-1pt]180:\raisebox{4pt}{$t^k$}}] {};
	\node (tst) at (-1,1) [circle,fill,scale=.5,label={[label distance=-4pt]135:$t^{k-2}st$}] {};
	\node (st) at (-2,0) [circle,fill,scale=.5,label={[label distance=-1pt]180:\raisebox{4pt}{$s^{k-1}t$}}] {};
	\node (ts) at (0,0) [circle,fill,scale=.5,label={[label distance=-1pt]0:\raisebox{4pt}{$t^{k-1}s$}}] {};
	\node (sts) at (-1,-1) [circle,fill,scale=.5,label={[label distance=-3pt]180:\raisebox{-16pt}{$s^{k-2}ts$}}] {};
	\node (s) at (-1,-3) [circle,fill,scale=.5,label={[label distance=-1pt]180:\raisebox{4pt}{$s^k$}}] {};
	\draw (sts) -- (ts) -- (tst) -- (t);
	\draw (s) -- (sts) -- (st) -- (tst);
\end{tikzpicture}}
\hspace{1pt}\raisebox{46pt}{$\cdots$}
\caption{The set $\CUP{n=2}{\infty}\ \CUP{m=2}{n}
\ {\overset{\infty}{\underset{\scriptstyle k=1}{\vrule width0pt height8pt depth2.3pt\smash{\bigcup}}}}\,P(m,n,k)$
where $P(m,n,k)$ is the partial order on $\{s,t\}^k$ in $\C mn$.}
\label{fig:st}
\end{figure}

\noindent
subject to the congruence conditions
\[
j\equiv d+1\pmod{n-1},\quad s\equiv d+1\pmod{m-1}.
\]
Such exponents exist by the Chinese remainder theorem, since
$d+1\equiv1\pmod{d\hspace{1pt}}$. The same argument applies symmetrically to
\[
ts=ts^j\le t^dt^rs=t^{d+1}s=ts^{d+1}=ts^ks^d\le t^ss=ts.
\]
If $s^{kd+1}t=st$, then $s^{(k+1)d+1}t=s^dst=st$, and similarly for $t^{kd+1}s=ts$.
\end{proof}

\begin{lem}\label{lem:genkey}
Let \mth{$2\le m\le n$.} The following hold in \mth{$\C mn$,} where \mth{$|w|$} denotes the length of \mth{$w$:}

\medskip\noindent
\nit{(i)}~\mth{$swt=st^{|w|+1}$} and \mth{$tws=ts^{|w|+1}$} for all \mth{$w\in\{s,t\}^*$,} where
\mth{$\{s,t\}^*=\{s,t\}^+\cup\{w\mathop||w|=0\}$,}

\medskip\noindent
\nit{(ii)}~\mth{$swtw's=st^{|w|+|w'|+1}s$} and \mth{$twsw't=ts^{|w|+|w'|+1}t$} for all \mth{$w,w'\in\{s,t\}^*$.}
\end{lem}

\noindent
\textit{Proof}.
By Lemma~\ref{lem:key}(iii), $swt\le st^{|w|+1}=s^{|w|+1}t\le swt$ and $tws\le t^{|w|+1}s=ts^{|w|+1}\le tws$.
Similarly,
\[
swtw's\le st^{|w|+|w'|+1}s=s^{|w|+1}t^{|w'|+1}s=s^{|w|+1}ts^{|w'|+1}\le swtw's,
\]

\vspace{-18pt}
\[
\pushQED{\qed}
twsw't\le t^{|w|+1}st^{|w'|+1}=ts^{|w|+1}t^{|w'|+1}=ts^{|w|+|w'|+1}t\le twsw't.\qedhere
\popQED
\]

\smallskip
The next proposition, which shows that $\C mn$ is finite, follows from
Lemmas~\ref{lem:key} and~\ref{lem:genkey}.

\begin{prop}\label{prop:st}
Let \mth{$2\le m\le n$.}
Every equivalence class in \mth{$\C mn$} has at least one element in the set
\[
W(m,n):=S(m{-}1)\cup T(n{-}1)\cup U(d\hspace{1pt})\subseteq\{s,t\}^+,\mbox{ where }d=\gcd(m-1,n-1)\mbox{ and}
\]

\vspace{-.78\baselineskip}
\[
S(k)=\{s,s^2,\dots,s^k\},\quad T(k)=\{t,t^2,\dots,t^k\},\quad
U(k)=\smash{\CUP{j=1}{k}}\{s^jt,s^jts,t^js,t^jst\}.
\]
\end{prop}
\vspace{.1\baselineskip}

\begin{prop}\label{prop:edges}
All edges hold in Figure~\nit{\ref{fig:st}.}
\end{prop}

\noindent
\textit{Proof}.
Lemma~\ref{lem:key}(iii) together with $s\le t$ imply that the following inequalities belong to $P(m,n,k)$ for all $k\ge1$:
\[
\pushQED{\qed}
s^k\le s^{k-2}ts\le st^{k-1}=s^{k-1}t\le t^{k-2}st\le t^k\mbox{\ \ and\ \ }s^{k-2}ts\le t^{k-1}s=
ts^{k-1}\le t^{k-2}st.\qedhere
\popQED
\]

\begin{lem}\label{lem:counterexamples}
Figure~\nit{\ref{fig:st}} is complete\nit, for the following hold in general in \mth{$\C mn$} for all \mth{$2\le m\le n$:}

\medskip\smallskip
{
\centering
\renewcommand{\tabcolsep}{2pt}
\begin{tabular}{rlcrl}
\nit{(i)}&\mth{$t^k\not\le tsw$} for all \mth{$w\in\{s,t\}^*$} and \mth{$k\ge1$,}&\quad\quad&
\nit{(iii)}&\mth{$tw\not\le sw'$} for all \mth{$w,w'\in\{s,t\}^*$,}\\[6pt]
\nit{(ii)}&\mth{$stw\not\le s^k$} for all \mth{$w\in\{s,t\}^*$} and \mth{$k\ge1$,}&&
\nit{(iv)}&\mth{$wt\not\le w's$} for all \mth{$w,w'\in\{s,t\}^*$.}\\
\end{tabular}

}
\end{lem}

\begin{proof}
Let $f_1\equiv1$ and $f_2\equiv2$ on $\{1,2\}$.
The pairs $f_1\le\ident$, $\ident\le f_2$, and $f_1\le f_2$
prove (i), (ii), and (iii), respectively.
The endofunctions $s,t$ on $\{1,2,3\}$ defined by $s(1)=s(2)=1$, $s(3)=3$, $t(1)=1$, and $t(2)=t(3)=3$
satisfy $wt(2)=3$ and $w's(2)=1$ for all $w,w'\in\{s,t\}^*$,
proving (iv).
\end{proof}

\begin{dfn}\label{dfn:shift}
For integers \mth{$u<v$,} let \mth{$\sigma_{\!u,\,v}$} denote the circular shift operator on
\mth{$\{u,\dots,v\}$,} defined by
\[
\sigma_{u,\,v}(k)=\begin{cases}k+1,&\mbox{if }u\le k<v,\\u,&\mbox{if }k=v.\end{cases}
\]
\end{dfn}

\begin{lem}\label{lem:completeness}
Figure~\nit{\ref{fig:st}} displays every inequality that holds in general in \mth{$\C mn$}
for arbitrary \mth{$2\le m\le n$.}
\end{lem}

\begin{proof}
Fix $2\le m\le n$. By Lemma~\ref{lem:key}(iv), we have $s^{k_1}t=s^{k_2}t$ and $t^{k_1}s=t^{k_2}s$ in $\C mn$
for all $k_1,k_2\ge1$ such that $k_1\equiv k_2\pmod{d\hspace{1pt}}$. Hence, by Proposition~\ref{prop:st},
it suffices to prove that the following hold in general in $\C mn$:
$s^{k_1}\not\le s^{k_2}$ for all $k_1\neq k_2$ in $\{1,\dots,m-1\}$,
$t^{k_1}\not\le t^{k_2}$ for all $k_1\neq k_2$ in $\{1,\dots,n-1\}$, and
$w\not\le w'$ for all $w,w'\in\{s,t\}^+$ such that $|w|\not\equiv|w'|\pmod{d\hspace{1pt}}$.

Let $z\in\{2,\dots,n-1\}$.
Since $z-1<n-1$, there exists a prime power $p^r$ such that
$p^r\mid n-1$ and $p^r\nmid z-1$.
Define the poset $L$ as follows:

\begin{center}
\begin{tikzpicture}[scale=.4]
	\node (1) at (0,0) [circle,fill,scale=.5,label={[label distance=1pt]270:\hspace{.5pt}$1$}] {};
	\node (2) at (-4,3) [circle,fill,scale=.5,label={[label distance=-1pt]180:$2$}] {};
	\node (3) at (-2,3) [circle,fill,scale=.5,label={[label distance=-1pt]180:$3$}] {};
	\node (4) at (0,3) [circle,fill,scale=.5,label={[label distance=-1pt]180:$4$}] {};
	\node (dots) at (3,3) [label={[label distance=0pt]180:$\cdots$}] {};
	\node (d) at (4,3) [circle,fill,scale=.5,label={[label distance=-1pt]0:$p^r+1$}] {};
	\draw (1) -- (2);
	\draw (1) -- (3);
	\draw (1) -- (4);
	\draw (1) -- (d);
\end{tikzpicture}
\end{center}

\noindent
Set $s\equiv1$ on $L$, $t(1)=1$, and $t(\ell)=\sigma_{2,\,p^r+1}(\ell)$ for $2\le\ell\le p^r+1$. %each $k\in\{2,\dots,p^r+1\}$.
Clearly $s$ and $t$ preserve the order on $L$, $s\le t$, and $s^m=s$.
For any $k\ge1$,
\[
t^k=t\ \Longleftrightarrow\ \sigma_{2,\,p^r+1}^k=\sigma_{2,\,p^r+1}\ \Longleftrightarrow\ k\equiv1\hspace{-6pt}\pmod{p^r}.
\]
Hence $t^n=t$, and since $t(p^r+1)=2<t^z(p^r+1)$, we also have $2=t^z((t^z)^{-1}(2))<t((t^z)^{-1}(2))$.
Thus $t\not\le t^z$ and $t^z\not\le t$ in $\C mn$.
By monotonicity of $t$ and $t^n=t$, it follows that $t^{z_1}\not\le t^{z_2}$
for all distinct $z_1\neq z_2$ in $\{1,\dots,n-1\}$.
A dual argument yields the same conclusion for powers of $s$ in $\{1,\dots,m-1\}$.

Now let $L=\{1,\dots,d\}$ be the trivial (discrete) poset, where $d=\gcd(m-1,n-1)$, and define $s=t=\sigma_{1,\,d}$ on $L$.
Then for all $w,w'\in\{s,t\}^+$, $w=w'\ \Longleftrightarrow\ |w|\equiv|w'|\pmod{d\hspace{1pt}}$,
and since $w=w'$ iff $w\le w'$ on $L$, it follows that $w\not\le w'$ holds in general in $\C mn$
whenever $|w|\not\equiv|w'|\pmod{d\hspace{1pt}}$ if $s$ and $t$ satisfy Definition~\ref{dfn:chittenden}.
Monotonicity and $s\le t$ are immediate.
Let $\ell=\operatorname{lcm}(m-1,n-1)$.
Then for all $k\ge1$, $\sigma^k_{1,\,d}=\sigma_{1,\,d}\ \Longleftrightarrow\ k\equiv1\pmod{d\hspace{1pt}}$,
hence $\vrule width0pt height12pts^m=\smash{s^{\frac{\ell}{n-1}d+1}}=s$ and $t^n=\smash{t^{\frac{\ell}{m-1}d+1}}=t$.
\end{proof}

\begin{thm}\label{thm:st}
Let \mth{$2\le m\le n$.}  The set \mth{$W(m,n)$} in Proposition\nit{~\ref{prop:st}} is comprised of exactly one
representative from each equivalence class in \mth{$\C mn$.} Moreover\nit, each such representative is in reduced form\nit.
\end{thm}

\begin{proof}
Proposition~\ref{prop:st} and Lemma~\ref{lem:completeness} imply the first statement, from which the second follows
because the components $S(k)$, $T(k)$, and $U(k)$ of $W(m,n)$ are each a union over the index set $\{1,\dots,k\}$.
\end{proof}

The next corollary is immediate.

\begin{cor}\label{cor:st}
\mth{$\C mn\subseteq\C nn\subseteq\C {n+1}{n+1}$} for all \mth{$2\le m\le n$.}
\end{cor}

The following multiplication algorithm holds by Theorem~\ref{thm:st} and Lemmas~\ref{lem:key}
and~\ref{lem:genkey}.

\begin{prop}\label{prop:multiplication}
Let \mth{$2\le m\le n$,} \mth{$d=\gcd(m-1,n-1)$,} \mth{$w,w'\in\C mn$,} and \mth{$\lambda=|w|+|w'|$.}
For \mth{$p=r_{m-1}(\lambda)$,} \mth{$q=r_{n-1}(\lambda)$,}
\mth{$u=r_d(\lambda-2)$,} and \mth{$v=r_d(\lambda-1)$} where \mth{$r_k(x)=1+\big((x-1)\ \mbox{mod}\ k\big)$,}
we have

\bigskip
{
\centering
\renewcommand{\tabcolsep}{2pt}
}\\
\footnotesize\raisebox{5pt}{\makebox[0pt]{$\phantom{\overset{6}{\underset{k=1}{\bigcup}}}$}$\C35$}\\
\end{tabular}

\caption{The Hasse diagram of $\C mn$ for selected $(m,n)$.}
\label{fig:hasse}
\end{figure}

Theorem~\ref{thm:chittenden} generalizes to $\C mn$ as follows.

\begin{prop}\label{prop:idempotence}
Let \mth{$2\le m\le n$} and \mth{$d=\gcd(m-1,n-1)$.} In \mth{$\C mn$,}
\mth{$(st)^k=st$} and \mth{$(ts)^k=ts$} where
\[
k=\begin{cases}\frac{d}{2}+1,&\mbox{if }d\mbox{ is even\hspace{.5pt}\nit,}\\d+1,&\mbox{if }d\mbox{ is odd\!\;\nit.}\end{cases}
\]
\end{prop}

\begin{proof}
Suppose first that $d$ is even. Since $w=(st)^{d/2}$ begins with $s$ and $w'=st$
ends with $t$, Proposition~\ref{prop:multiplication} gives
$(st)^{d/2+1}=s^{[(d+2)-1]\bmod d}t=st$.
If $d$ is odd, the same proposition yields $(st)^{d+1}=st$.
The identity $(ts)^k=ts$ follows by duality.
\end{proof}

Let $2\le m\le n$ and $d=\gcd(m-1,n-1)$. In general, $s^jt\neq st$ in $\C mn$
for $2\le j\le d$ by Theorem~\ref{thm:st}.
For even (resp.\ odd) $d$ this implies $(st)^k=s^{2k-1}t\neq st$ for $2\le k\le\frac
d2\ (\mbox{resp.\ }2\le k\le\hspace{-.2pt}\frac{d+1}2\hspace{-.5pt})$.
Now suppose $d$ is odd and $d\ge3$.
Then $2k-1\equiv2,4,\dots,d-1\pmod{d\hspace{1pt}}$ for
$k=\frac{d+3}2,\frac{d+5}2,\dots,\frac{d+d}2$, and hence, by Proposition~\ref{prop:multiplication},
$(st)^k=s^{2k-1}t\in\{s^2t,s^4t,\dots,s^{d-1}t\}$.
By Theorem~\ref{thm:st}, each of these differs from $st$, so $(st)^k\neq st$.
Since the case $d=1$ is trivial, we conclude that the answer to Question~\ref{qst:chittenden}
is $I(m,n)=k$, where $k$ is the exponent given in Proposition~\ref{prop:idempotence}.

We now turn to global orderings and collapses. 

\begin{prop}\label{prop:global_orderings}
Let \mth{$2\le m\le n$.} Every global ordering of \mth{$\C mn$} is a global collapse of \mth{$\C mn$.}
\end{prop}

\begin{proof}
This holds provided each of the following conditions is true in $\C mn$:

\medskip\noindent
(i)~$ws^jw'\le ws^kw'\,\Longleftrightarrow\,ws^jw'=ws^kw'$ for all $w,w'\in\{s,t\}^*$
and $1\le j,k\le m-1$,

\medskip\noindent
(ii)~$wt^jw'\le wt^kw'\,\Longleftrightarrow\,wt^jw'=wt^kw'$ for all $w,w'\in\{s,t\}^*$
and $1\le j,k\le n-1$,

\medskip\noindent
(iii)~$s^j\le t^k\,\Longleftrightarrow\,s^jt=s^kt$ for all $1\le j\le m-1$, $1\le k\le n-1$,

\medskip\noindent
(iv)~$t^k\le s^j\,\Longleftrightarrow\,t^k=s^j$ for all $1\le j\le m-1$, $1\le k\le n-1$.

\bigskip\noindent
(i)~By monotonicity, the partial order on $\C mn$ respects left-multiplication by words $w\in\{s,t\}^*$.
For arbitrary $1\le j,k\le m-1$,
\[
(ws^jw'\le ws^kw'\mbox{ for all }w,w'\in\{s,t\}^*)\ \Longleftrightarrow\ (s^jw'\le s^kw'
\mbox{ for all }w'\in\{s,t\}^*)\ \Longleftrightarrow\ s^j\le s^k.
\]
Hence it suffices to show that $s^j\le s^k\,\Longrightarrow\,s^k\le s^j$.
Suppose $s^j\le s^k$ with $1\le j,k\le m-1$, and let $p=(m-1)-j$.
Then $s^k=s^{k+(m-1)}=s^{k+p}(s^j)\le s^{k+p}(s^k)=s^{2k+p}$.
Since $s=s^{(m-1)-j+1+j}=s^{p+1}s^j$, it follows for all $q\ge2$ that
\[
s^{qk+(q-1)p}=s^{qk+(q-1)p-1}(s^{p+1}s^j)=s^{q(k+p)}(s^j)\le s^{q(k+p)}(s^k)=s^{(q+1)k+qp}.
\]
Thus $s^k\le s^{qk+(q-1)p}$ for all $q\ge1$.
Let $r=k+p-1$. For $q=m-1$,
\[
qk+(q-1)p=(m-1)k+[(m-1)-1][(m-1)-j]=r(m-1)+j.
\]
Hence $s^k\le s^{r(m-1)+j}=s^j$. The proof of (ii) is similar.

(iii)~Let $1\le j\le m-1$ and $1\le k\le n-1$.
Left- and right-multiplying $s^j\le t^k$ by $s^{\ell}$ and $t$, respectively, yields
$s^{\ell}s^jt\le s^{\ell}t^{k+1}$, where $\ell=\operatorname{lcm}(m-1,n-1)$.
By (i) and Lemma~\ref{lem:key}, this is equivalent to
\[
s^jt=s^{\ell+j}t=s^{\ell}t^{k+1}=s^{\ell+k}t=s^kt.
\]
Right-multiplying by $t^{\ell-1}$ gives $s^j=s^{\ell+j}\le s^jt^{\ell}=s^kt^{\ell}\le t^{\ell+k}=t^k$.

(iv)~If $t^k\le s^j$, then $s^k\le t^k\le s^j$, hence $s^k=t^k=s^j$ by (i). The reverse implication is trivial.
\end{proof}

\begin{dfn}\label{dfn:parity}
We say that a word \mth{$w\in\{s,t\}^*$} is \define{even} if its length is even. We call the component of even
operators in the Hasse diagram of \mth{$\C33$} the \define{even component;} its \define{odd component} and
\define{odd words} are defined similarly\nit.
Two subcomponents of this diagram are said to be \define{congruent} when one
is a horizontal translation of the other. We say that an equation \mth{$w_1=w_2$} between words
\mth{$w_1,w_2\in\{s,t\}^+$} is \define{even} if \mth{$|w_1|-|w_2|$} is even\nit, and call an equivalence class in
Figure~\nit{\ref{fig:chittenden}} \define{even} if and only if all of its equations are even\nit;
\define{odd} equations and classes are defined similarly\nit.
\end{dfn}

\begin{lem}\label{lem:odd}
Every even \nit(\!\!\;resp\nit.\,odd\,\nit)\!\!\; equivalence class in Figure~\nit{\ref{fig:chittenden}}
containing exactly two equations collapses two congruent subcomponents of the Hasse diagram of
\mth{$\C33$} into two points \nit(\!\!\;resp\nit.\,a single point\nit{).}
\end{lem}

\begin{proof}
The claim is immediate for even classes.
For odd classes with exactly two equations, we have
$u_1=v_2\,\Longleftrightarrow\,u_2=v_1$,
where $u_1,u_2$ and $v_1,v_2$ are endpoints of congruent
subcomponents satisfying $u_1\le u_2$ and $v_1\le v_2$ in general.
Hence $v_2=u_1\le u_2=v_1$ and $u_2=v_1\le v_2=u_1$, which yields the result.
\end{proof}

\begin{lem}\label{lem:hasse}
All implications hold in Figure~\nit{\ref{fig:chittenden}.}
\end{lem}

\begin{proof}
Implications within equivalence classes hold after a single right- or left-multiplication by $s$ or $t$,
followed, if necessary, by one or two applications of $s^3=s$, $t^3=t$, $(st)^2=st$, $(ts)^2=ts$, $tts=tss$,
$sst=stt$. For example, left-multiplying $t=tts$ by $t$ gives $tt=ttts=ts$, and left-multiplying $ssts=ttst$ by $s$
gives $sts=sssts=sttst=sstst=sst$.

Every arrow in Figure~\ref{fig:chittenden} that emanates from a class containing exactly two equations follows from
Lemma~\ref{lem:odd} together with the Hasse diagram of $\C33$.
All arrows involving four-equation classes hold by similar reasoning; in particular, the least such class
collapses the two diamond-shaped subcomponents into a single point.
Finally, to get the remaining two arrows, right-multiply $s=ss$ by $t$ and $t=tt$ by $s$.
\end{proof}

We are now ready to prove the analogue of Theorem~\ref{thm:monoids} for $\C33$.

\begin{thm}\label{thm:collapses}
Order-preserving endofunctions \mth{$s\le t$} satisfying \mth{$s^3=s$} and \mth{$t^3=t$} on a
poset \mth{$(L,\le)$}\!\!\; always generate one of \mth{$52$} distinct monoids under composition
\!\!\;\nit(\!\!\;Figure~\nit{\ref{fig:global_collapses}\hspace{.5pt}),}
\!\!\;all of which occur\nit.
\end{thm}

\begin{proof}
The $27$ equivalence classes in Figure~\ref{fig:chittenden} together with the empty class represent $28$
possible global collapses.
Since Table~$1$ contains exactly $21$ blank white cells, at most $21$ new ones can arise
by conjoining pairs of classes in Figure~\ref{fig:chittenden}.
Because each gray cell or black entry appears immediately below (resp.\ above) every teal entry
without (resp.\ with) a ``d'',
it follows that any conjunction of three distinct equivalence classes
not in $\{2,2$d$\}$ is equivalent to a conjunction of two or fewer classes.
By duality, Table~$1$ shows that the only remaining three-class conjunctions to consider
are the $28$ that conjoin class~$2$ with a pair in
$E=\{2\mbox{d},6\mbox{d},9,11,12\mbox{d},13\mbox{d},15,16\mbox{d}\}$.
Since every pair in $\{6\mbox{d}\}\times(E\setminus\{2\mbox{d}\})$ either has a black entry or a gray cell in
Table~$1$, and no teal entries appear below row $8$d, all but the seven conjunctions
of $2$ and $2$d with some class in $E\setminus\{2\mbox{d}\}$ can be ruled out.
Because $2\mbox{d}\ge e$ for
$e\in\{6\mbox{d},12\mbox{d},13\mbox{d},16\mbox{d}\}$, the three remaining cases $e\in\{9,11,15\}$
in Figure~\ref{fig:global_collapses} are the only ones possible.
Since the conjunction of any pair in $\{9,11,15\}$ is equivalent to a single class, we conclude that
$\C33$ admits at most $52$ global collapses.
A computer verification confirms that each of these $52$ possibilities occurs for some poset
on at most eight points.
\end{proof}

\begin{cor}\label{cor:collapses}
\mth{$\C23$} has exactly \mth{$24$} global collapses\nit.
\end{cor}

\begin{proof}
By Theorem~\ref{thm:collapses}, the only possible global collapses of $\C22$ are the $16$ in
Figure~\ref{fig:global_collapses} that satisfy $s=ss$ and $t=tt$.\footnote{Every global collapse of $\K$ in
Figure~\ref{fig:monoids} except~$10$ and~$10$d corresponds to one of these $16$ global collapses ($10$
and $10$d each collapse to the discrete case when $\ident$ is removed).}
The only possible global collapses of $\C23$ are these $16$
together with the eight in Figure~\ref{fig:global_collapses} that satisfy $s=ss$ and $t\neq tt$.
A computer verification confirms that each of these $24$ possibilities occurs for some poset
on at most seven points.
\end{proof}

\begin{figure}
\centering\scriptsize
\renewcommand{\arraystretch}{1}
\renewcommand{\tabcolsep}{1pt}
\begin{tabular}{R{16pt}|C{16pt}|C{16pt}|C{16pt}|C{16pt}|C{16pt}|C{16pt}|C{16pt}|C{16pt}|C{16pt}|C{16pt}|C{16pt}|C{16pt}|C{16pt}|C{16pt}|C{16pt}|C{16pt}|C{16pt}|C{16pt}|C{16pt}|C{16pt}|C{16pt}|C{16pt}|C{16pt}|C{16pt}|C{16pt}|}
\multicolumn{1}{c}{\diagbox[width=18pt, height=11pt, linewidth=.5pt]{\raisebox{-2.5pt}{\hspace*{1pt}$e_1$}}
{\raisebox{4.5pt}{\hspace{-9pt}$e_2$}}}&\multicolumn{1}{c}{2}&\multicolumn{1}{c}{2d}&\multicolumn{1}{c}{3}&
\multicolumn{1}{c}{3d}&\multicolumn{1}{c}{4}&
\multicolumn{1}{c}{5}&\multicolumn{1}{c}{6}&\multicolumn{1}{c}{6d}&
\multicolumn{1}{c}{7}&\multicolumn{1}{c}{7d}&\multicolumn{1}{c}{8}&
\multicolumn{1}{c}{8d}&\multicolumn{1}{c}{9}&\multicolumn{1}{c}{10}&
\multicolumn{1}{c}{11}&\multicolumn{1}{c}{12}&\multicolumn{1}{c}{12d}&
\multicolumn{1}{c}{13}&\multicolumn{1}{c}{13d}&\multicolumn{1}{c}{14}&
\multicolumn{1}{c}{14d}&\multicolumn{1}{c}{15}&\multicolumn{1}{c}{16}&
\multicolumn{1}{c}{16d}&\multicolumn{1}{c}{17}\\[-.3pt]\hhline{~|-|-|-|-|-|-|-|-|-|-|-|-|-|-|-|-|-|-|-|-|-|-|-|-|-}
\vrule width0pt height9pt1\kern2pt&\cellcolor{gray!25}&\cellcolor{gray!25}&6&6d&9&11&\cellcolor{gray!25}&\cellcolor{gray!25}&12&12d&13&13d&\cellcolor{gray!25}&15&\cellcolor{gray!25}&\cellcolor{gray!25}&\cellcolor{gray!25}&\cellcolor{gray!25}&\cellcolor{gray!25}&16&16d&\cellcolor{gray!25}&\cellcolor{gray!25}&\cellcolor{gray!25}&18\\\hhline{~|-|-|-|-|-|-|-|-|-|-|-|-|-|-|-|-|-|-|-|-|-|-|-|-|-}
\multicolumn{1}{c}{\vrule width0pt height9pt}&2&&6&\textcolor{LimeGreen}{\textit{6d}}&\textcolor{LimeGreen}{\textit{9}}&\textcolor{LimeGreen}{\textit{11}}&\cellcolor{gray!25}&&12&\textcolor{LimeGreen}{\textit{12d}}&13&\textcolor{LimeGreen}{\textit{13d}}&&\textcolor{LimeGreen}{\textit{15}}&&\cellcolor{gray!25}&&\cellcolor{gray!25}&&16&\textcolor{LimeGreen}{\textit{16d}}&&\cellcolor{gray!25}&&18\\
\end{tabular}%
\caption{\vrule width0pt height10ptAll general implications between equivalence classes
of the $\binom{12}{2}$ equations in $\C33$.}
\label{fig:chittenden}
\end{figure}

\begin{figure}
\centering
%
};
\node[anchor=300] (a) at (-2.4,.5) {};
\node[anchor=120] (b) at (-2.4,.5) {};
\node[anchor=240] (c) at (2.4,.5) {};
\node[anchor=60] (d) at (2.4,.5) {};
\node[anchor=240] (e) at (3.2,5) {};
\node[anchor=60] (f) at (3.2,5) {};
\node[anchor=300] (g) at (-2.2,4.4) {};
\node[anchor=120] (h) at (-2.2,4.4) {};
\node[anchor=340] (i) at (2.2,7.5) {};
\node[anchor=160] (j) at (2.2,7.5) {};
\node[anchor=270] (k) at (-2.7,8.5) {};
\node[anchor=90] (l) at (-2.7,8.5) {};
\node[anchor=200] (m) at (-1.8,7.36) {};
\node[anchor=20] (n) at (-1.8,7.36) {};
\node[anchor=340] (o) at (4.9,4.8) {};
\node[anchor=160] (p) at (4.9,4.8) {};
\node[anchor=0] (u) at (1,15.1) {};
\node[anchor=180] (v) at (1,15.1) {};
\node[anchor=180] (w) at (-1,14.9) {};
\node[anchor=0] (x) at (-1,14.9) {};
\node[anchor=260] (aa) at (-3.3,7) {};
\node[anchor=80] (bb) at (-3.3,7) {};
\node[anchor=290] (cc) at (3.5,7) {};
\node[anchor=110] (dd) at (3.5,7) {};
\node[anchor=180] (gg) at (-1,13.2) {};
\node[anchor=0] (hh) at (-1,13.2) {};
\node[anchor=0] (ii) at (1,13.4) {};
\node[anchor=180] (jj) at (1,13.4) {};
\node[anchor=210] (mm) at (-4.4,11.3) {};
\node[anchor=30] (nn) at (-4.4,11.3) {};
\node[anchor=330] (oo) at (4.4,11.3) {};
\node[anchor=150] (pp) at (4.4,11.3) {};
\node[anchor=190] (qq) at (-4.8,10.04) {};
\node[anchor=10] (rr) at (-4.8,10.04) {};
\node[anchor=350] (ss) at (4.8,10.04) {};
\node[anchor=170] (tt) at (4.8,10.04) {};
\node[anchor=270] (uu) at (-3.6,14.5) {};
\node[anchor=90] (vv) at (-3.6,14.5) {};

%	0 -> 1
%	0 -> 2
%	0 -> 3
%	0 -> 4
%	0 -> 5
\draw (0) -- (1);
\draw[extend upper=1.2pt,line width=1pt] (0.west) -- (2.south east);
\draw[extended=.3pt,line width=1pt] (0.west) -- (3and4.197);
\draw[extended=.3pt,line width=1pt] (0.east) -- (3and4.343);
\draw[extend upper=1.2pt,line width=1pt] (0.east) -- (5.south west);

%	1 -> 6
%	1 -> 7
%	1 -> 8
%	1 -> 9
\path (1.west) edge[out=180,in=0,extended=.3pt] (6.344);
\path (1.west) edge[out=180,in=300,extended=.3pt] (a);
\path (b) edge[out=120,in=210,extended=.3pt] (7and8.south west);
\path (1.east) edge[out=0,in=240,extended=.3pt] (c);
\path (d) edge[out=60,in=330,extended=.3pt] (7and8.south east);
\path (1.east) edge[out=0,in=180,extended=.3pt] (9.196);

%	2 -> 14
%	2 -> 15
%	2 -> 17
%	2 -> 18
\path (2.east) edge[out=60,in=300,extended=.3pt] (14.east);
\path (2.east) edge[out=40,in=320,extended=.3pt,line width=1pt] (15.east);
\path (2.east) edge[out=50,in=310,extended=.3pt,line width=1pt] (17.east);
\draw[extended=.6pt,line width=1pt] (2.north east) -- (18.south west);

%	3 -> 7
%	4 -> 8
%	4 -> 10
%	4 -> 15
%	3 -> 17
%	3 -> 19
%	3 -> 20
%	4 -> 20
\draw (3and4.north) -- (3and4.south);
\draw[extended=.3pt] (3and4.158) -- (7and8.202);
\draw[extended=.3pt] (3and4.22) -- (7and8.338);
\draw[extended=.6pt,line width=1pt] (3and4.north east) -- (10.south west);
\draw[extended=1.2pt,line width=1pt] (3and4.22) -- (15.south east);
\draw[extended=.9pt,line width=1pt] (3and4.north west) -- (17.south east);
\draw[extend lower=1.2pt,line width=1pt] (3and4.158) -- (19.west);
\path (3and4.north west) edge[out=110,in=230,extended=.9pt,line width=1pt] (20.south west);
\path (3and4.north east) edge[out=70,in=310,extended=.9pt,line width=1pt] (20.south east);

%	5 -> 10
%	5 -> 12
%	5 -> 18
%	5 -> 19
\path (5.west) edge[out=140,in=220,extended=.3pt,line width=1pt] (10.west);
\path (5.west) edge[out=120,in=240,extended=.3pt] (12.west);
\draw[extended=.6pt,line width=1pt] (5.north west) -- (18.south east);
\path (5.west) edge[out=130,in=230,extended=.3pt,line width=1pt] (19.west);

%	6 -> 16
%	6 -> 21
%	6 -> 22
\path (6.east) edge[out=40,in=320,extended=.3pt] (16.east);
\path (6.east) edge[out=30,in=330,extended=.3pt] (21.east);
\path (6.east) edge[out=0,in=260,extended=.3pt] (aa);
\path (bb) edge[out=80,in=180,extended=.3pt] (22.west);

%	8 -> 11
%	8 -> 16
%	7 -> 21
%	7 -> 23
%	7 -> 24
%	8 -> 24
\draw (7and8.north) -- (7and8.south);
\path (7and8.north east) edge[out=60,in=180,extended=.3pt] (11.west);
\path (7and8.22) edge[out=160,in=300,extend lower=.4pt] (g);
\path (h) edge[out=120,in=350,extended=.3pt] (16.east);
\draw[extended=.4pt] (7and8.north west) -- (21.south east);
\path (7and8.158) edge[out=20,in=240,extend lower=.4pt] (e);
\path (f) edge[out=60,in=245,extended=.3pt] (23.south west);
\path (7and8.north west) edge[out=100,in=220,extend lower=.8pt,extend upper=.3pt] (24.south west);
\path (7and8.north east) edge[out=80,in=318,extend lower=.8pt,extend upper=.3pt] (24.south east);

%	9 -> 11
%	9 -> 22
%	9 -> 23
\path (9.west) edge[out=150,in=210,extended=.3pt] (11.west);
\path (9.west) edge[out=180,in=290,extended=.3pt] (cc);
\path (dd) edge[out=110,in=0,extended=.3pt] (22.east);
\path (9.west) edge[out=160,in=200,extended=.3pt] (23.west);

%	10 -> 13
%	10 -> 25
%	10 -> 27
\path (10.west) edge[out=120,in=240,extended=.3pt] (13.west);
\path (10.north west) edge[out=120,in=340,extend lower=.6pt,line width=1pt] (i);
\path (j) edge[out=160,in=340,extend upper=1pt,line width=1pt] (25and41.202);
\path (10.west) edge[out=130,in=230,extended=.3pt,line width=1pt] (27.west);

%	11 -> 26
%	11 -> 28
\path (11.north west) edge[out=120,in=0,extended=.3pt] (u);
\path (v) edge[out=180,in=340,extended=.4pt] (26and42.202);
\path (11.west) edge[out=150,in=210,extended=.3pt] (28.west);

%	12 -> 9
%	12 -> 13
%	12 -> 29
%	12 -> 30
\path (12.west) edge[out=140,in=220,extended=.3pt] (9.west);
\path (12.west) edge[out=160,in=200,extended=.3pt] (13.west);
\path (12.west) edge[out=150,in=210,extended=.3pt] (29.west);
\path (12.west) edge[out=170,in=190,extended=.3pt] (30.west);

%	13 -> 11
%	13 -> 31
%	13 -> 32
\path (13.west) edge[out=150,in=210,extended=.3pt] (11.west);
\path (13.west) edge[out=120,in=240,extended=.3pt] (31.west);
\path (13.west) edge[out=180,in=180,extended=.3pt] (32.west);

%	14 -> 6
%	14 -> 33
%	14 -> 34
%	14 -> 35
\path (14.east) edge[out=40,in=320,extended=.3pt] (6.east);
\path (14.east) edge[out=20,in=340,extended=.3pt] (33.east);
\path (14.east) edge[out=10,in=350,extended=.3pt] (34.east);
\path (14.east) edge[out=30,in=330,extended=.3pt] (35.east);

%	15 -> 25
%	15 -> 33
%	15 -> 36
\draw[extended=.6pt,line width=1pt] (15.north east) -- (25and41.193);
\path (15.east) edge[out=60,in=300,extended=.3pt] (33.east);
\path (15.east) edge[out=50,in=310,extended=.3pt,line width=1pt] (36.east);

%	16 -> 26
%	16 -> 37
\draw[extended=.3pt] (16.north east) -- (26and42.west);
\path (16.east) edge[out=30,in=330,extended=.3pt] (37.east);

%	17 -> 34
%	17 -> 36
%	17 -> 41
\path (17.east) edge[out=40,in=320,extended=.3pt] (34.east);
\path (17.east) edge[out=30,in=330,extended=.3pt,line width=1pt] (36.east);
\path (17.north east) edge[out=60,in=200,extended=.6pt,line width=1pt] (m);
\path (n) edge[out=20,in=210,extended=.6pt,line width=1pt] (25and41.338);

%	18 -> 25
%	18 -> 29
%	18 -> 35
%	18 -> 41
\draw[extended=.3pt,line width=1pt] (18.160) -- (25and41.195);
\draw[extended=.3pt] (18.east) -- (29.west);
\draw[extended=.3pt] (18.west) -- (35.east);
\draw[extended=.3pt,line width=1pt] (18.20) -- (25and41.345);

%	19 -> 27
%	19 -> 30
%	19 -> 41
\path (19.west) edge[out=150,in=210,extended=.3pt,line width=1pt] (27.west);
\path (19.west) edge[out=140,in=220,extended=.3pt] (30.west);
\draw[extended=.6pt,line width=1pt] (19.north west) -- (25and41.347);

%	20 -> 24
%	20 -> 27
%	20 -> 36
\path (20.10) edge[out=10,in=330,extended=.3pt] (24.336);
\draw[extended=.6pt,line width=1pt] (20.north east) -- (27.west);
\draw[extended=.6pt,line width=1pt] (20.north west) -- (36.east);

%	21 -> 37
%	21 -> 42
\path (21.east) edge[out=40,in=320,extended=.3pt] (37.east);
\path (21.north east) edge[out=60,in=180,extended=.3pt] (w);
\path (x) edge[out=0,in=220,extended=.3pt] (26and42.338);

%	22 -> 26
%	22 -> 42
\path (22.north east) edge[out=60,in=330,extended=.3pt] (26and42.south east);
\path (22.north west) edge[out=120,in=210,extended=.3pt] (26and42.south west);

%	23 -> 28
%	23 -> 42
\path (23.west) edge[out=140,in=220,extended=.3pt] (28.west);
\draw[extended=.3pt] (23.north west) -- (26and42.east);

%	24 -> 28
%	24 -> 37
\path (24.14) edge[out=0,in=240,extended=.3pt] (28.south west);
\path (24.166) edge[out=180,in=300,extended=.3pt] (37.south east);

%	25 -> 31
%	25 -> 38
%	25 -> 43
%	41 -> 43
%	41 -> 45
%	41 -> 48
\draw (25and41.north) -- (25and41.south);
\path (25and41.140) edge[out=24,in=212,extended=.3pt] (31.west);
\draw[extend lower=.3pt] (25and41.north west) -- (38.east);
\draw[extend upper=.3pt,line width=1pt] (25and41.165) -- (43.west);
\draw[extend upper=.3pt,line width=1pt] (25and41.15) -- (43.east);
\draw (25and41.north east) -- (45.west);
\path (25and41.40) edge[out=162,in=328,extended=.3pt] (48.south east);

%	26 -> 44
%	42 -> 44
\draw (26and42.north) -- (26and42.south);
\draw[extend upper=.5pt] (26and42.164) -- (44.south west);
\draw[extend upper=.5pt] (26and42.16) -- (44.south east);

%	27 -> 32
%	27 -> 43
\path (27.west) edge[out=120,in=240,extended=.3pt] (32.west);
\path (27.west) edge[out=120,in=0,extended=.3pt,line width=1pt] (43.east);

%	28 -> 44
\draw[extend lower=.6pt] (28.north west) -- (44.east);

%	29 -> 31
%	29 -> 39
%	29 -> 45
\path (29.west) edge[out=130,in=230,extended=.3pt] (31.west);
\draw (29.west) -- (39.east);
\path (29.west) edge[out=140,in=220,extended=.3pt] (45.west);

%	30 -> 23
%	30 -> 32
%	30 -> 45
\path (30.west) edge[out=130,in=230,extended=.3pt] (23.west);
\path (30.west) edge[out=140,in=220,extended=.3pt] (32.west);
\path (30.west) edge[out=120,in=240,extended=.3pt] (45.west);

%	31 -> 40
%	31 -> 46
\path (31.north west) edge[out=120,in=0,extended=.3pt] (ii);
\path (jj) edge[out=180,in=340,extended=.3pt] (40and50.200);
\path (31.west) edge[out=140,in=220,extended=.3pt] (46.west);

%	32 -> 28
%	32 -> 46
\path (32.west) edge[out=130,in=230,extended=.3pt] (28.west);
\path (32.west) edge[out=120,in=240,extended=.3pt] (46.west);

%	33 -> 16
%	33 -> 38
%	33 -> 47
\path (33.east) edge[out=10,in=350,extended=.3pt] (16.east);
\path (33.east) edge[out=50,in=310,extended=.3pt] (38.east);
\path (33.east) edge[out=6,in=354,extended=.3pt] (47.east);

%	34 -> 21
%	34 -> 47
%	34 -> 48
\path (34.east) edge[out=50,in=310,extended=.3pt] (21.east);
\path (34.east) edge[out=40,in=320,extended=.3pt] (47.east);
\path (34.east) edge[out=60,in=300,extended=.3pt] (48.east);

%	35 -> 38
%	35 -> 39
%	35 -> 48
\path (35.east) edge[out=70,in=290,extended=.3pt] (38.east);
\draw (35.east) -- (39.west);
\path (35.east) edge[out=50,in=310,extended=.3pt] (48.east);

%	36 -> 43
%	36 -> 47
\path (36.east) edge[out=60,in=180,extended=.3pt,line width=1pt] (43.west);
\path (36.east) edge[out=60,in=300,extended=.3pt] (47.east);

%	37 -> 44
\draw[extend lower=.6pt] (37.north east) -- (44.west);

%	38 -> 40
%	38 -> 49
\draw[extend lower=.4pt,extend upper=.3pt] (38.north east) -- (40and50.south west);
\path (38.east) edge[out=40,in=320,extended=.3pt] (49.east);

%	39 -> 22
%	39 -> 40
%	39 -> 50
\path (39.west) edge[out=180,in=270,extended=.3pt] (k);
\path (l) edge[out=90,in=180,extended=.3pt] (22.west);
\path (39.west) edge[out=180,in=270,extended=.1pt] (40and50.south west);
\path (39.east) edge[out=0,in=296,extended=.3pt] (40and50.south east);

%	40 -> 26
%	40 -> 51
%	50 -> 42
%	50 -> 51
\draw (40and50.north) -- (40and50.south);
\draw (40and50.160) -- (26and42.202);
\path (40and50.north west) edge[out=120,in=190,extended=.3pt] (51.west);
\draw (40and50.20) -- (26and42.338);
\path (40and50.north east) edge[out=60,in=350,extended=.3pt] (51.east);

%	43 -> 46
%	43 -> 49
\draw[extend upper=.4pt] (43.east) -- (46.south west);
\draw[extend upper=.4pt] (43.west) -- (49.south east);

%	45 -> 46
%	45 -> 50
\path (45.west) edge[out=130,in=230,extended=.3pt] (46.west);
\draw[extend lower=.4pt] (45.north west) -- (40and50.south east);

%	46 -> 51
\draw[extend lower=.6pt] (46.north west) -- (51.east);

%	47 -> 37
%	47 -> 49
\path (47.east) edge[out=50,in=310,extended=.3pt] (37.east);
\path (47.east) edge[out=60,in=300,extended=.3pt] (49.east);

%	48 -> 49
%	48 -> 50
\path (48.east) edge[out=50,in=310,extended=.3pt] (49.east);
\path (48.north east) edge[out=60,in=180,extended=.3pt] (gg);
\path (hh) edge[out=0,in=220,extended=.3pt] (40and50.340);

%	49 -> 51
\draw[extend lower=.8pt] (49.north east) -- (51.west);

%	51 -> 44
\draw (51) -- (44);

\end{tikzpicture}
\caption{\vrule width0pt height10ptThe $52$ global collapses of $\C33$, ordered by containment.
Bold denotes $\C22$.}
\label{fig:global_collapses}
\end{figure}

\section*{Acknowledgments}
Francesco~Ciraulo kindly provided valuable help with Sections~\ref{sec:intro}-\ref{sec:pseudo},
sharing three counterexamples I initially thought sufficed to prove Theorem~\ref{thm:localic}
(I later realized a fourth is needed).
%Specifically, Ciraulo pointed out that the Sierpiński locale disproves $i=ici$, any open
%interval in the frame of opens of the reals disproves $ici=ci$, and the sublocale
%$X_{\neg\neg}$ in \cite{2025_ciraulo} disproves $ic=ic{-}{-}$.
Gro\mbox{-}Tsen's answer at \url{https://mathoverflow.net/a/501577/5090} was helpful for the
proof of Theorem~\ref{thm:localic}, and Hagen~von~Eitzen's answer at
\url{https://math.stackexchange.com/a/5069214/32209} was useful for the proof
of Lemma~\ref{lem:completeness}.
ChatGPT (OpenAI, 2025) assisted by generating \texttt{C} code and providing stylistic feedback
on the prose.

\bigskip
\advance\baselineskip by -.7pt

\interlinepenalty=10000

\bibliography{kuratowski_g}

\end{document}